\documentclass[10pt,a4paper]{amsart}

\usepackage{amsmath}
\usepackage{amsthm}
\usepackage{enumerate}
\usepackage{fancyhdr}
\usepackage{amssymb}
\usepackage{graphicx}
\usepackage[all]{xy}
\usepackage{hyperref}
\usepackage{aliascnt}
\usepackage{stmaryrd}
\usepackage{braket}
\usepackage{xcolor}
\usepackage{mdframed}
\usepackage{mathtools}
\usepackage{verbatim}
\usepackage{txfonts}
\usepackage{setspace}   

\newtheorem{lma}{Lemma}[section]

\newaliascnt{thmCt}{lma}
\newtheorem{thm}[thmCt]{Theorem}
\aliascntresetthe{thmCt}

\newaliascnt{corCt}{lma}
\newtheorem{cor}[corCt]{Corollary}
\aliascntresetthe{corCt}

\newaliascnt{propCt}{lma}
\newtheorem{prop}[propCt]{Proposition}
\aliascntresetthe{propCt}

\newtheorem*{thm*}{Theorem}
\newtheorem*{cor*}{Corollary}
\newtheorem*{prop*}{Proposition}

\theoremstyle{definition}

\newaliascnt{prgCt}{lma}
\newtheorem{prg}[prgCt]{}
\aliascntresetthe{prgCt}

\newaliascnt{dfnCt}{lma}
\newtheorem{dfn}[dfnCt]{Definition}
\aliascntresetthe{dfnCt}

\newaliascnt{rmkCt}{lma}
\newtheorem{rmk}[rmkCt]{Remark}
\aliascntresetthe{rmkCt}

\newaliascnt{rmksCt}{lma}

\aliascntresetthe{rmksCt}

\newaliascnt{ntnCt}{lma}
\newtheorem{ntn}[ntnCt]{Notation}
\aliascntresetthe{ntnCt}

\newaliascnt{qstCt}{lma}

\aliascntresetthe{qstCt}

\newaliascnt{prblCt}{lma}

\aliascntresetthe{prblCt}

\newaliascnt{exaCt}{lma}

\aliascntresetthe{exaCt}

\newcommand{\W}{\mathrm{W}}

\newcommand{\T}{\mathbb{T}}
\newcommand{\N}{\mathbb{N}}
\newcommand{\Q}{\mathbb{Q}}
\newcommand{\R}{\mathbb{R}}
\newcommand{\Z}{\mathbb{Z}}
\newcommand{\K}{\mathrm{K}}

\DeclareMathOperator{\her}{her}
\DeclareMathOperator{\Inv}{Inv}

\DeclareMathOperator{\Ban}{Ban}

\newcommand{\CatCa}{C^*}
\newcommand{\CatCasr}{C^*_{\text{sr}1}}
\newcommand{\CatPoM}{\mathrm{PoM}}
\newcommand{\CatoM}{\mathrm{Mon}_\leq}
\newcommand{\CatM}{\mathrm{Mon}}

\newcommand{\CatW}{\mathrm{W}}
\newcommand{\CatQ}{\mathrm{Q}}
\DeclareMathOperator{\sr}{sr}
\DeclareMathOperator{\Gr}{Gr}
\DeclareMathOperator{\Aff}{Aff}

\DeclareMathOperator{\Lat}{Lat}

\DeclareMathOperator{\AbGp}{AbGp}
\DeclareMathOperator{\Cu}{Cu}

\DeclareMathOperator{\NCCW}{NCCW}
\DeclareMathOperator{\1NCCW}{1-NCCW}
\DeclareMathOperator{\AI}{AI}
\DeclareMathOperator{\A}{A}
\DeclareMathOperator{\AH}{AH}
\DeclareMathOperator{\Ell}{Ell}
\DeclareMathOperator{\AF}{AF}
\DeclareMathOperator{\Lsc}{Lsc}

\DeclareMathOperator{\supp}{supp}
\DeclareMathOperator{\Hom}{Hom}

\newcommand{\hooklongrightarrow}{\lhook\joinrel\longrightarrow}

\begin{document}
\onehalfspacing
\author{Laurent Cantier}
\title{A unitary Cuntz semigroup for $\CatCa$-algebras of stable rank one}

\address{Laurent Cantier,
Departament de Matem\`{a}tiques \\
Universitat Aut\`{o}noma de Barcelona \\
08193 Bellaterra, Barcelona, Spain
}
\email[]{lcantier@mat.uab.cat}

\thanks{The author was supported by MINECO through the grant BES-2016-077192 and partially supported by the grants MDM-2014-0445 and MTM-2017-83487 at the Centre de Recerca Matem\`atica in Barcelona.}

\begin{abstract}
We introduce a new invariant for $\CatCa$-algebras of stable rank one that merges the Cuntz semigroup information together with the $\K_1$-group information. This semigroup, termed the $\Cu_1$-semigroup, is constructed as equivalence classes of pairs consisting of a positive element in the stabilization of the given $\CatCa$-algebra together with a unitary element of the unitization of the hereditary subalgebra generated by the given positive element. We show that the $\Cu_1$-semigroup is a well-defined continuous functor from the category of $\CatCa$-algebras of stable rank one to a suitable codomain category that we write $\Cu^\sim$. Furthermore, we compute the $\Cu_1$-semigroup of some specific classes of $\CatCa$-algebras. Finally, in the course of our investigation, we show that we can recover functorially $\Cu$, $\K_1$ and $\K_*:=\K_0\oplus\K_1$ from $\Cu_1$. 
 \end{abstract}
\maketitle

\section{Introduction}

The Elliott classification program aims to find a complete invariant for nuclear separable simple $\CatCa$-algebras. The original version of this invariant, written $\Ell(A)$, is based on $\K$-theoretical information together with tracial data. As up to now, adding up decades of research, this invariant has provided satisfactory results for simple, separable, unital, nuclear, $\mathcal{Z}$-stable $\CatCa$-algebras satisfying the Universal Coefficient Theorem assumption. (See, among many others, \cite{GLN15}, \cite{EGLN15}, and \cite{TWW17}.) On the other hand, the Cuntz semigroup has recently appeared to be a key tool to recover regularity properties of a (not necessarily simple) $\CatCa$-algebra. As a matter of fact, it has been proved that the Cuntz semigroup of $\mathcal{C}(\T)\otimes A$ is naturally isomorphic to $\Ell(A)$, for any unital, simple, nuclear, finite, $\mathcal{Z}$-stable $\CatCa$-algebra $A$ (see \cite{ADPS17}).

Classification of non-simple C*-algebras has had an important resurgence in recent years. Whenever considering non-simple $\CatCa$-algebras, the Cuntz semigroup, written $\Cu$, seems to be a good candidate itself for classification. For instance, it has been shown that the Cuntz semigroup classifies any (unital) inductive limits of one-dimensional non-commutative CW complexes whose $\K_1$-group is trivial (see \cite{RobNCCW1}). More concretely, the Cuntz semigroup entirely captures the complete lattice $\Lat(A)$ of ideals of any $\CatCa$-algebra $A$, since we have a natural lattice isomorphism between $\Lat(A)\simeq \Lat(\Cu(A))$, where $\Lat(\Cu(A))$ denotes the set of ideals of $\Cu(A)$. (See \cite[Proposition 5.1.10]{APT14}.) However, a main limitation of the Cuntz semigroup lies within the fact that it fails to capture any $\K_1$ information whatsoever.

In this paper, we introduce a unitary version of the Cuntz semigroup, denoted by $\Cu_1$, for $\CatCa$-algebras of stable rank one. This construction incorporates the $\K_1$ groups of the $\CatCa$-algebra and its ideals to overcome this lack of information in the original construction of the Cuntz semigroup. We here establish the basic functorial properties of this construction.
More concretely we show that:

The $\Cu_1$-semigroup is a continuous functor from the category of $\CatCa$-algebras with stable rank one, that we denote $\CatCasr$, to a certain subcategory of semigroups, written $\Cu^\sim$, modeled after the category $\Cu$ of abstract Cuntz semigroups.
\begin{thm}
The functor $\Cu_1: \CatCasr\longrightarrow \Cu^{\sim}$ is continuous. More precisely, given an inductive system $(A_i,\phi_{ij})_{i\in I}$ in $\CatCasr$, then:
\[
\Cu^\sim-\lim\limits_{\longrightarrow}(\Cu_1(A_i),\Cu_1(\phi_{ij}))\simeq \Cu_1(\CatCasr-\lim\limits_{\longrightarrow}((A_i,\phi_{ij}))).
\]
\end{thm}

We then recover functorially the $\K_*$-group from the $\Cu_1$-semigroup as follows:
\begin{thm}
There exists a functor  \vspace{-0,70cm}\[
	\begin{array}{ll}
	\hspace{0cm} H_*:\Cu^\sim_{u}\longrightarrow \AbGp_{u}\\
	\hspace{0,5cm} (S,u)\longmapsto (\Gr(S_{c}),S_{c},u)\\
		\hspace{1,05cm} \alpha \longmapsto \Gr(\alpha_{c})
	\end{array}
\] 
 that yields a natural isomorphism $\eta_*:H_*\circ\Cu_{1,u}\simeq \K_*$.
\end{thm}

This paper is organized as follows: 
In a first part, we construct our invariant, for $\CatCa$-algebras of stable rank one. We show that it is an ordered monoid that satisfies the order-theoretic axioms (O1)-(O4) introduced in \cite{CEI08}. We also find a suitable category, called the category $\Cu^\sim$, and prove that $\Cu_1$ is a well-defined continuous functor.

Then, we give an alternative picture of our invariant, making use of the lattice of ideals of the $\CatCa$-algebra, in order to compute the $\Cu_1$-semigroup of some classes of $\CatCa$-algebras, such as the simple case, $\AF$, and some $\AI$ and $\A\!\T$ algebras. 

Finally, we explicitly define the notion of recovering an invariant from another and how one can recover classifying results. We then see that we can recover $\Cu$, $\K_1$ and also $\K_*$ from $\Cu_1$, to conclude that $\Cu_1$ is a complete invariant for the class of unital $\AH_d$ algebras with real rank zero.

We mention that this article is part of a twofold work. The author has been investigating further on the unitary Cuntz semigroup in \cite{C20b}, studying its ideal structure and exactness properties.

\textbf{Acknowledgments.} The author is indebted to Ramon Antoine and Francesc Perera for suggesting the construction of such an invariant, as this work was part of my PhD. I am grateful for their patience and many fruitful discussions on the Cuntz semigroup and details about the continuity of the $\Cu_1$-semigroup. The author also wishes to thank the referee for his/her thorough revision and numerous comments that have helped to greatly improve the manuscript.

\section{Preliminaries}

\subsection{The Cuntz semigroup}
We recall some definitions and properties on the Cuntz semigroup of a $\CatCa$-algebra. More details can be found in \cite{APT14}, \cite{APT09}, \cite{CEI08}, \cite{T19}.

\begin{prg}\textbf{(The Cuntz semigroup of a $\CatCa$-algebra.)}
\label{dfn:Cusg}
Let $A$ be a $\CatCa$-algebra. We denote by $A_+$ the set of positive elements. Let $a$ and $b$ be in $A_+$. We say that $a$ is Cuntz subequivalent to $b$, and we write $a\lesssim_{\Cu} b$, if there exists a sequence $(x_n)_{n\in\N}$ in $A$ such that $a=\lim\limits_{n\in\N}x_nbx_n^*$. After antisymmetrizing this relation, we get an equivalence relation over $A_+$, called Cuntz equivalence, denoted by $\sim_{\Cu}$.

Let us write $\Cu(A):= (A\otimes\mathcal{K})_+/\!\!\sim_{\Cu}$, that is, the set of Cuntz equivalence classes of positive elements of $A\otimes\mathcal{K}$. Given $a\in (A\otimes\mathcal{K})_+ $ we write $[a]$ for the Cuntz equivalence class of $a$. This set is equipped with an addition as follows: let $v_1$ and $v_2$ be two isometries in the multiplier algebra of $A\otimes\mathcal{K}$, such that $v_1v_1^*+v_2v_2^*=1_{M(A\otimes\mathcal{K})}$. Consider the $^*$-isomorphism $\psi: M_2(A\otimes\mathcal{K})\longrightarrow A\otimes\mathcal{K}$ given by $\psi(\begin{smallmatrix} a & 0\\ 0 & b \end{smallmatrix})=v_1av_1^*+v_2bv_2^*$, and we write $a \oplus b:=\psi(\begin{smallmatrix} a & 0\\ 0 & b \end{smallmatrix})$. 
For any $[a],[b]$ in $\Cu(A)$, we define $[a]+[b]:=[a\oplus b]$ and $[a]\leq [b]$ whenever $a\lesssim_{\Cu} b$. In this way $\Cu(A)$ is a semigroup called \emph{the Cuntz semigroup of $A$}. 

For any $^*$-homomorphism $\phi:A\longrightarrow B$, one can define $\Cu(\phi):\Cu(A)\longrightarrow\Cu(B)$, a semigroup map, by $[a]\longmapsto [(\phi\otimes id_\mathcal{K})(a)]$. Hence, we get a functor from the category of $\CatCa$-algebras into a certain subcategory of the category $\CatPoM$ of positively ordered monoids, called the category $\Cu$, that we describe next. 
\end{prg}

\begin{dfn}
\label{dfn:auxiliary}
Let $(S,\leq)$ be an ordered semigroup. An \emph{auxiliary relation} on $S$ is a binary relation $\prec$ such that:

(i) For any $a,b\in S$ such that $a\prec b$ we have $a\leq b$.

(ii) For any $a,b,c,d\in S$ such that $a\leq b\prec c \leq d$ we have $a\prec d$.
\end{dfn}

\begin{prg}\textbf{(The category $\Cu$.)}
\label{dfn:llCU}
Let $(S,\leq)$ be a positively ordered semigroup and let $x,y$ in $S$. We say that $x$ is \emph{way-below} $y$, and we write $x\ll y$ if, for all increasing sequences $(z_n)_{n\in\N}$ in $S$ that have a supremum, if $\sup\limits_{n\in\N} z_n\geq y$, then there exists $k$ such that $z_k\geq x$. This is an auxiliary relation on $S$, called the \emph{compact-containment relation} and sometimes referred to as the \emph{way-below} relation. In particular $x\ll y$ implies $x\leq y$ and we say that $x$ is a \emph{compact element} whenever $x\ll x$. 

We say that $S$ is an abstract $\Cu$-semigroup if it satisfies the following order-theoretic axioms: 

$\,\,$(O1): Every increasing sequence of elements in $S$ has a supremum. 

$\,\,$(O2): For any $x\in S$, there exists a $\ll$-increasing sequence $(x_n)_{n\in\N}$ in $S$ such that $\sup\limits_{n\in\N} x_n= x$.

$\,\,$(O3): Addition and the compact containment relation are compatible.

$\,\,$(O4): Addition and suprema of increasing sequences are compatible.

A \emph{$\Cu$-morphism} between two $\Cu$-semigroups is a positively ordered monoid morphism that preserves the compact containment relation and suprema of increasing sequences. 

The category $\Cu$ of abstract Cuntz semigroups is the subcategory of $\CatPoM$ whose objects are $\Cu$-semigroups and morphisms are $\Cu$-morphisms. 
\end{prg}

\begin{prg}\textbf{(Properties of the Cuntz semigroup.)}
\label{prg:latticecu}
Let $S$ be a $\Cu$-semigroup. We say that $S$ is \emph{countably-based} if there exists a countable subset $B\subseteq S$ such that for any $a,a'\in S$ such that $a'\ll a$, then there exists $b\in B$ such that $a'\leq b \ll a$. An element $u\in S$ is called an \emph{order-unit} of $S$ if for any $x\in S$, there exists $n\in\overline{\N}:=\N\sqcup\{\infty\}$ such that $x\leq nu$.
A countably-based $\Cu$-semigroup has a largest element or, equivalently, it is singly-generated as an ideal -for instance, by its largest element-. Let us also mention that if $A$ is a separable $\CatCa$-algebra, then $\Cu(A)$ is countably-based. In fact, its largest element, that we write $\infty_A$, can be explicitly constructed as $\infty_{A}=\sup\limits_{n\in\N}n[s_A]$, where $s_A$ is any strictly positive element (or full positive) in $A$. A fortiori, $[s_A]$ is an order-unit of $\Cu(A)$.

A notion of ideals in the category $\Cu$ has been considered in several places; we refer the reader to \cite[\S 5.1.6]{APT14} for more details. We recall that for any $\Cu$-semigroup $S$ and any $x\in S$, the ideal generated by $x$ is $I_x:=\{y\in S \mid y\leq \infty x\}$. For any $\CatCa$-algebra A, the assignment $I\longmapsto \Cu(I)$ defines a lattice isomorphism between the lattice $\Lat(A)$ of closed two-sided ideals of $A$ and the lattice $\Lat(\Cu(A))$ of ideals of $\Cu(A)$. In fact, $a$ is a full positive element in $I$ if and only if $[a]$ is a full element in $\Cu(I)$. In this case, we have $\Cu(I_a)=I_{[a]}$.
\end{prg}

\subsection{The stable rank one context}
\label{prop:viewequi}
As mentioned before, we work with $\CatCa$-algebras of stable rank one. In this context, Cuntz subequivalence of positive elements admits a nicer description easier to work with. Let us shortly explicit this alternative picture and we refer the reader to \cite[Proposition 4.3 - \S 6]{Ortega-Rordam-Thiel}, \cite[Proposition 1]{CES11} and \cite{P-Z} for more details. 

Let $A$ be a $\CatCa$-algebra. We recall that an \emph{open projection} is a projection $p\in A^{**}$ such that $p$ belongs to the strong closure of the hereditary subalgebra $A_p:=pA^{**}p\cap A$ of $A$. These open projections are in one-to-one correspondence with the hereditary subalgebras of $A$. For any positive element $a$ of $A$, we shall write $\her(a):=\overline{aAa}$, the hereditary subalgebra of $A$ generated by $a$ and call the \emph{support projection of $a$}, the (unique) open projection $p_a\in A^{**}$ such that $\her(a)= A_{p_a}$. We recall that $p_a:=SOT-\lim a^{1/n}$.

We also recall that two open projections $p,q\in A^{**}$ are \emph{Peligrad-Zsid\'o equivalent}, and we write $p \sim_{PZ} q$ if there exists a partial isometry $v\in A^{**}$ such that $p=v^*v, q=vv^*, vA_p\subseteq A, A_qv^*\subseteq A$. We say that $p\lesssim_{PZ} q$ if there exists an open projection $p'\in A^{**}$ such that $p\sim_{PZ} p'\leq q$; see \cite[Definition 1.1]{P-Z}. 

Suppose now that $A$ has stable rank one. Then $a\lesssim_{\Cu}b$ if and only if there exists $x\in A$ such that $xx^*=a$ and $x^*x\in \her(b)$. This is in turn equivalent to saying that $p_a \lesssim_{PZ} p_b$.
In this case, for any partial isometry $\alpha \in A^{**}$ that realizes the Peligrad-Zsid\'o equivalence between $p_a$ and $p_b$, we have an explicit injection as follows: 
\vspace{-0,4cm}\[
    \begin{array}{ll}
    		\hspace{-1,5cm} \theta_{ab,\alpha}: \her(a) \hooklongrightarrow \her(b) \\
    		d \longmapsto \alpha^*d\alpha
    \end{array}
\]

The next proposition is similar to \cite[Proposition 3.3]{Ortega-Rordam-Thiel} and \cite[Theorem 1.4]{P-Z}. For the sake of completeness we will give a proof in this slightly different picture.

\begin{prop}
\label{prop:MvNpz}
Let $A$ be a $\CatCa$-algebra. Let $a$ be in $A_+$ and let $p\in A^{**}$ be its support projection. Let $\alpha$ be a partial isometry in $A^{**}$ such that $p=\alpha\alpha^*$ and $q:=\alpha^*\alpha$ is an open projection of $A^{**}$. Set $x:=a^{1/2}\alpha$. Then $p\sim_{PZ}q$ if and only if $x$ belongs to A. In this case, $q=p_{x^*x}$.
\end{prop}

\begin{proof}
The forward implication is coming from the definition of the Peligrad-Zsid\'o equivalence itself. 

Conversely, let us suppose that $x:=a^{1/2}\alpha$ belongs to $A$. Let $d$ be in $aAa$. Then there exists $\delta_d$ in $A$ such that $d=a\delta_d a$. Now observe that $\alpha^*d=\alpha^*a^{1/2}a^{1/2}\delta_d a$ belongs to $A$. We obtain that $\alpha^*aAa\subseteq A$, and hence $\alpha^*\overline{aAa}\subseteq A$, that is, $\alpha^*A_p\subseteq A$. 
Now since $p$ is a support projection and $q=\alpha^*p\alpha$, we deduce that $q$ is a support projection and moreover $\alpha^*A_p\alpha=A_q$.
Finally, observe that $\alpha A_q=\alpha A_q\alpha^*\alpha=A_p\alpha$ and that $(\alpha^*A_p)^*=A_p\alpha$, so $\alpha A_q\subseteq A$. 
We conclude that $p\sim_{PZ}q$ and by construction $q=p_{x^*x}$.
\end{proof}

\begin{lma}
\label{lma:invarK1inj}
Let $A$ be a $\CatCa$-algebra with stable rank one and let $a$ and $b$ be contractions in $A_+$ such that $a\lesssim_{\Cu} b$. Let $\alpha$ and $\beta$ be in $A^{**}$ such that they both realize the Peligrad-Zsid\'o subequivalence of $p_a\lesssim_{PZ}p_b$. For any $u \in \mathcal{U}(\her(a)^\sim)$, we have
\vspace{-0,1cm}\[
[\theta^\sim_{ab,\alpha}(u)]_{\K_1(\her(b)^\sim)}=[\theta^\sim_{ab,\beta}(u)]_{\K_1(\her(b)^\sim)}
\vspace{-0,1cm}\]
where $\theta^\sim_{ab,\alpha}$ (resp $\theta^\sim_{ab,\beta}$) is the unitized morphism of $\theta_{ab,\alpha}$ in \autoref{prop:viewequi}.
\end{lma}

\begin{proof}
Since $a$ and $b$ are fixed elements, we shall write $\theta_\alpha$ instead of $\theta_{ab,\alpha}$ (respectively $\theta_\beta$ for $\theta_{ab,\beta}$). 
Consider the injections given by $\alpha$ and $\beta$ as in \autoref{prop:viewequi}.  
Define $x:=a^{1/2}\alpha$ and $y:=a^{1/2}\beta$. We have $x,y\in A$. We first consider elements of $aAa$ and the result will follow by continuity. Rewrite $\theta_{\alpha}$ and $\theta_{\beta}$:
\vspace{-0,1cm}\[
	\begin{array}{ll} 
		\theta_{\alpha}: aAa \hooklongrightarrow \overline{bAb}   \,\,\,\,\,\,\,\,\,\,\,\,\,\,\,\,\,\,\,\, \,\,\,\,\,\,\,\,\,\,\,\,\,\,\,\,\,\,\, \theta_{\beta}: aAa \hooklongrightarrow \overline{bAb} \\
		\hspace{0,76cm} a\delta a \longmapsto x^*a^{1/2}\delta a^{1/2}x \,\,\,\,\,\,\,\,\,\,\,	\hspace{1,2cm} a\delta a \longmapsto y^*a^{1/2}\delta a^{1/2}y
	\end{array}
\]

Let $u$ be a unitary element of $\her(a)^\sim$. There exists a pair $(u_0,\lambda)$ with $u_0\in\her(a)$ and $\lambda\in\T$ such that $u=u_0+\lambda$.

Let $0<\epsilon<2$. Since $\her(a)=\overline{aAa}$, we can find $\delta\in A$ such that $\Vert u_0-a\delta a\Vert \leq\epsilon/3$. We write $M:=\Vert \delta \Vert$ and we set $\epsilon':=\epsilon/(6M)$. On the one hand, observe that $\Vert a^{1/2} \Vert\leq 1$ and hence we easily get that $\Vert a^{1/2}\delta a^{1/2}\Vert\leq M$. 


On the other hand, since $a=xx^*=yy^*$, by \cite[Lemma 2.4]{CES11} we know there exists a unitary element $u_\epsilon$ of $\her(b)^\sim$ such that $\Vert y-xu_\epsilon\Vert\leq\epsilon'$ (equivalently $\Vert u_\epsilon^*x^*-y^*\Vert \leq\epsilon'$). 
Now, we compute:
\begin{align*}
\Vert u_\epsilon^*\theta_{\alpha}^\sim(a\delta a+\lambda)u_\epsilon-\theta_{\beta}^\sim(a\delta a+\lambda)\Vert & = \Vert u_\epsilon^*x^*a^{1/2}\delta a^{1/2}xu_\epsilon - y^*a^{1/2}\delta a^{1/2}y \Vert \\
&\leq \Vert u_\epsilon^*x^*a^{1/2}\delta a^{1/2}xu_\epsilon - y^*a^{1/2}\delta a^{1/2}xu_\epsilon \Vert\\
&\,\,\,\,\,\,\,\, +\Vert y^*a^{1/2}\delta a^{1/2}xu_\epsilon - y^*a^{1/2}\delta a^{1/2}y \Vert \\
& \leq \Vert u_\epsilon^*x^*-y^*\Vert\,\Vert a^{1/2}\delta a^{1/2}\Vert \,\Vert xu_\epsilon\Vert+ \Vert y-xu_\epsilon\Vert \,\Vert a^{1/2}\delta a^{1/2}\Vert \,\Vert y^*\Vert\\
& \leq \epsilon' M+\epsilon' M\\
&\leq \epsilon/3.
\end{align*}
Combining the fact that $u$ and $a\delta a+\lambda$ are close up to $\epsilon/3$ with the fact that $\theta_{\alpha}^\sim$ and $\theta_{\beta}^\sim$ are contractive maps, we conclude that $\Vert u_\epsilon^*\theta_{\alpha}^\sim(u)u_\epsilon-\theta_{\beta}^\sim(u)\Vert \leq\epsilon<2$. On the other hand, it is well-known that unitary elements that are close enough (i.e. $\Vert u-v\Vert < 2$) are homotopic. We conclude that $u_\epsilon^*\theta_{\alpha}^\sim(u)u_\epsilon\sim_h\theta_{\beta}^\sim(u)$ and the result follows.
\end{proof}

\section{The \texorpdfstring{$\Cu_1$}{Cu1} semigroup}
\label{sec:Def Cu1}
In this section we define the invariant and establish its first properties. The unitary Cuntz semigroup consists of classes of pairs of element $(a,u)$, where $a$ is a positive element of $A\otimes\mathcal{K}$ and $u$ is a unitary of $\her(a)^\sim$, under a suitable equivalence relation, written $\sim_1$, that is built using the Cuntz subequivalence to compare positive elements and using \autoref{lma:invarK1inj} to compare unitary elements. Our main result here focuses on the continuity of this invariant. 
\subsection{The \texorpdfstring{$\lesssim_1$}{Unitary Cuntz} binary relation}
\label{rmk:canoinj2}

Let $A$ be a $\CatCa$-algebra of \emph{stable rank one}. Let $a,b\in A_+$ and let $u,v$ be unitary elements of $\her(a)^\sim$ and $\her(b)^\sim$ respectively. 

We say that $(a,u)$ is \emph{unitarily Cuntz subequivalent} to $(b,v)$, and we write $(a,u)\lesssim_1(b,v)$ if, \vspace{-0,1cm}\[\left\{\begin{array}{ll} \,a\lesssim_{\Cu} b \\ \,[\theta_{ab,\alpha}^\sim(u)]=[v] \text{ in } \K_1(\her(b)^\sim)\end{array}\right.\vspace{-0cm}\]
where $\theta_{ab,\alpha}$ is the injection given by a partial isometry $\alpha$ as constructed in \autoref{prop:viewequi}.

\begin{lma}
The relation $\lesssim_1$ is reflexive and transitive. 
\end{lma}

\begin{proof}
Reflexivity of $\lesssim_1$ follows from the fact that $\lesssim_{\Cu}$ is reflexive and that $id_{\her(a)}=\theta_{aa,p_a}$. 

Now let $a,b$ and $c$ be in $A_+$ and let $u_{a},u_{b}$ and $u_{c}$ be unitary elements of $\her(a)^\sim$, $\her(b)^\sim$ and $\her(c)^\sim$ respectively. Assume that $(a,u_a)\lesssim_1(b,u_b)$ and $(b,u_b)\lesssim_1(c,u_c)$.
By hypothesis, we know that $a \lesssim_{\Cu} b$ and $b\lesssim_{\Cu} c$. Since $A$ has stable rank one, there exist $x,y\in A$ such that $a=xx^*$, $b=yy^*$, $x^*x\in\her(b)$ and $y^*y\in\her(c)$.
Let us consider the polar decompositions of $x$ and $y$. That is, $x=a^{1/2}\alpha, y=b^{1/2}\beta$, for some partial isometries $\alpha, \beta$ of $A^{**}$. 
Using \autoref{prop:viewequi}, we get $p_a=\alpha\alpha^*\sim_{PZ}\alpha^*\alpha\leq p_b$ and also $p_b\sim_{PZ}\beta^*\beta\leq p_c$. We set $q_a:=\alpha^*\alpha, q_b:=\beta^*\beta$. One can check that $\gamma:=\alpha\beta$ is a partial isometry of $A^{**}$ and that $p_a=\gamma\gamma^*$.

Let us write $z:=a^{1/2}\gamma$. Observe that $zz^*=a$ and also $z=x\beta$. We hence compute that $z^*z=\beta^*x^*x\beta\in\her(c)$. We deduce that $zz^*=a$ and $z^*z\in\her(c)$. 
By \cite[Proposition 2.12]{APT09} we may write $x:=u^*(x^*x)^{1/3}$ for some element $u$ of $A$. Since $(x^*x)\in A_{p_b}$ and $\beta^*A_{p_b}\subseteq A$, we deduce that $\beta^*x^*$ is in $A$, and hence $z\in A$.

Using \autoref{prop:MvNpz}, we obtain that $q_c:=\gamma^*\gamma$ is the support projection of $z^*z$ and is Peligrad-Zsid\'o equivalent to $p_a$. 
Finally, \autoref{lma:invarK1inj} tells us that $\theta_{ac,\gamma}=\theta_{bc,\beta}\circ\theta_{ab,\alpha}$ is one of the morphisms described in \autoref{prop:viewequi}, from which the transitivity of $\lesssim_1$ follows.
\end{proof}

\subsection{Standard maps}
\label{rmk:canoinj}
We have seen that for any unitary element $u$ of $\her(a)^{\sim}$ and any partial isometry $\alpha \in A^{**}$ such that $p_a=\alpha\alpha^*$ and $\alpha^*\alpha\leq p_b$, the $\K_1$-class of $\theta^\sim_{ab,\alpha}(u)$ does not depend on the $\alpha$ chosen. In the sequel, whenever $a\lesssim_{\Cu}b$, we will refer to the maps $\theta^\sim_{ab,\alpha}$ as \emph{standard maps} and will rewrite them as $\theta_{ab}$. In particular, whenever $a\leq b$ observe that the canonical inclusion map $i$ is a standard map.
Also, notice that every standard morphism between $a$ and $b$ gives rise to the same group morphism at the $\K_1$-level, that we will denote by $\chi_{ab}$. That is, $\chi_{ab}:=\K_1(\theta_{ab}):\K_1(\her(a))\longrightarrow \K_1(\her(b))$.

\subsection{The \texorpdfstring{$\Cu_1$}{Cu1}-semigroup}
\label{dfn:Cu1}

We construct the \emph{unitary Cuntz semigroup} of a stable rank one $\CatCa$-algebra in a similar fashion to the original Cuntz semigroup, using the $\lesssim_1$ relation and the standard maps. 

Let $A$ be a $\CatCa$-algebra of \emph{stable rank one}. By antisymmetrizing the $\lesssim_1$ relation, we define an equivalence relation $\sim_1$ on the set of pairs $(a,u)$ where $a\in (A\otimes\mathcal{K})_+$ and $u\in \mathcal{U}(\her(a)^\sim)$. The equivalence relation $\sim_1$ is called the \emph{unitary Cuntz equivalence} and we denote by $[(a,u)]$ the equivalence class of $(a,u)$. We construct the unitary Cuntz semigroup of $A$ as follows:
\vspace{-0cm}\[\Cu_1(A):= \{(a,u) \mid a\in (A\otimes\mathcal{K})_+ , u\in \mathcal{U}(\her(a)^\sim)\}/\!\!\sim_1.\]

The set $\Cu_1(A)$ naturally inherits a partial order induced by the relation $\lesssim_1$. More concretely, for any two $[(a,u)],[(b,v)]\in\Cu_1(A)$, we say that $[(a,u)]\leq [(b,v)]$ if and only if $(a,u)\lesssim_1 (b,v)$.

The addition on $\Cu_1(A)$ is defined component-wise and mimics the construction of the addition in $\Cu(A)$: given any two elements $a,b \in (A\otimes\mathcal{K})_+$, we know that $a \oplus b:=\psi(\begin{smallmatrix} a & 0\\ 0 & b \end{smallmatrix})$ is a positive element of $A\otimes\mathcal{K}$, where $\psi:M_2(A\otimes\mathcal{K})\simeq A\otimes\mathcal{K}$ (see \autoref{dfn:Cusg}). Given unitary elements $u\in\mathcal{U}(\her(a)^\sim)$ and $v\in\mathcal{U}(\her(b)^\sim)$, we first observe that their respective scalar part can be assumed to be equal without loss of generality, since the equivalence relation $\sim_1$ identifies unitary elements up to homotopy equivalence. Besides, $\psi(\begin{smallmatrix} \her(a) & 0\\ 0 & \her(b) \end{smallmatrix})\subseteq \her(a\oplus b)$ and hence $u\oplus v:=\psi^\sim(\begin{smallmatrix} u & 0\\ 0 & v \end{smallmatrix})$ is a unitary element of $\her(a\oplus b)^\sim$. We conclude that $[(a,u)]+[(b,v)]:=[(a\oplus b, u\oplus v)]$ is a well-defined element in $\Cu_1(A)$.

Therefore, we obtain a partially ordered monoid $(\Cu_1(A),+,\leq)$ whose neutral element is $[(0_A,1_\mathbb{C})]$ and the proof is left to the reader. By positive elements, we mean elements that are greater or equal to the neutral element. It is easy to see that the positive elements of $\Cu_1(A)$ are those of the form $([a,1])$ for some $a\in (A\otimes\mathcal{K})_+$ and thus $\Cu_1(A)$ is in general not positively ordered.
We now show that $(\Cu_1(A),\leq)$ satisfies the order-theoretic axioms (O1)-(04) mentioned in \autoref{dfn:llCU}.

\begin{prop}
\label{cor:existencelift}
Let $A$ be a $\CatCa$-algebra of stable rank one. Let $(a_n)_n$ be a sequence in $A_+$ such that $a_n\lesssim_{\Cu} a_m$, for any $n\leq m$. Let $a\in A_+$ be any representative of $\sup\limits_{n}[a_n]\in\Cu(A)$ obtained from axiom (O1) and the stable rank one hypothesis. Then, for any unitary element $u\in \her(a)^\sim$, there exists a unitary element $u_n$ in $\her(a_n)^\sim$ for some $n\in\N$ such that $[(a_n,u_n)]\leq [(a,u)]$ in $\Cu_1(A)$.
\end{prop}

\begin{proof}
For any $n\in\N$, consider $b_n:=(a-1/n)_+$. It is well-known that $([b_n])_n$ is a $\ll$-increasing sequence in $\Cu(A)$ whose supremum is $[a]$; see e.g. \cite[Proposition 2.61]{T19}. Also, it is not hard to check that $\AbGp-\lim\limits_{\longrightarrow}(\K_1(\her(b_n)),\chi_{b_nb_m})\simeq (\K_1(\her(a)),\chi_{b_na})$.
Since we are in the category $\AbGp$, for any $[u]\in\K_1(\her(a))$, we can find $n$ and $[u_n]\in \K_1(\her(b_n))$ such that $\chi_{b_na}([u_n])=[u]$. Since $A$ has stable rank one, then so does $\her(b_n)^\sim$. Hence using $\K_1$-surjectivity (see \cite[Theorem 2.10]{R87}), we can find a unitary element $u_n$ of $\her(b_n)^\sim$ whose $\K_1$-class is $[u_n]$.
On the other hand, $([a_m])_m$ is an increasing sequence in $\Cu(A)$ whose supremum is $[a]$ and hence there exists $m\in\N$ such that $[b_n]\leq [a_m]$ in $\Cu(A)$. So we can consider the unitary element $\theta_{b_na_m}(u_n)$ in $\her(a_m)^\sim$. By transitivity of $\lesssim_1$, we obtain that $\chi_{a_ma}([\theta_{b_na_m}(u_n)])=\chi_{a_ma}\circ\chi_{b_na_m}([u_n])=\chi_{b_na}([u_n])=[u]$ and the result follows.
\end{proof}

\begin{lma}\label{lma:O1}
Let $A$ be a $\CatCa$-algebra of stable rank one. Then any increasing sequence $([(a_n,u_n)])_{n\in\N}$ in $\Cu_1(A)$ has a supremum $[(a,u)]$ in $\Cu_1(A)$. In particular, $[a]=\sup\limits_n [a_n]$ in $\Cu(A)$ and there exists $n\in\N$ large enough such that $[u]=\chi_{a_na}([u_n])$ in $\K_1(\her(a))$.
\end{lma}

\begin{proof}
Let $([(a_n,u_n)])_{n\in\N}$ be an increasing sequence in $\Cu_1(A)$. Then $([a_n])_{n\in\N}$ is an increasing sequence in $\Cu(A)$. By (O1) in $\Cu(A)$, the sequence $([a_n])_{n\in\N}$ has a supremum $[a]$ in $\Cu(A)$. Now, let $n\leq m$. Since $[(a_n,u_n)]\leq [(a_m,u_m)]$, we get that $\chi_{a_na_m}([u_n])=[u_m]$. By transitivity of $\lesssim_1$, we obtain that $\chi_{a_na}([u_n])=\chi_{a_ma}([u_m])$ in $\K_1(\her(a))$. Write $[u]:= \chi_{a_na}([u_n])$. We deduce that $[(a,u)]\geq [(a_n,u_n)]$ in $\Cu_1(A)$ for any $n\in\N$.

Let us check that $[(a,u)]$ is in fact the supremum of the sequence $([(a_n,u_n)])_n$. Let $[(b,v)]\in\Cu_1(A)$ such that $[(b,v)]\geq [(a_n,u_n)]$ for every $n\in\N$. Since $[a]=\sup\limits_{n\in\N}[a_n]$, we have $[b]\geq [a]$ in $\Cu(A)$. Using transitivity of $\lesssim_1$, the following diagram is commutative: 
\[
\xymatrix{
\K_1(\her(a_n)^\sim)\ar[rd]^{\chi_{a_na}}\ar@/^/[rrd]^{\chi_{a_nb}}\ar[dd]_{\chi_{a_na_m}} & &  \\
& \K_1(\her(a)^\sim)\ar[r]_{\chi_{ab}} & \K_1(\her(b)^\sim) \\
\K_1(\her(a_m)^\sim)\ar[ru]^{\chi_{a_ma}}\ar@/_/[rru]_{\chi_{a_mb}} & &
} 
\]
Hence for every $n$ and $m$ in $\N$, we have $\chi_{a_nb}([u_n])=\chi_{a_mb}([u_m])=\chi_{ab}([u])$ in $\K_1(\her(b))$. We deduce that $\chi_{ab}([u])=[v]$ in $\K_1(\her(b))$ and hence $[(a,u)]\leq [(b,v)]$. 
\end{proof}

\begin{prop}
\label{prop:llCU1}
Let $A$ be a $\CatCa$-algebra of stable rank one and let $[(a,u)],[(b,v)]\in\Cu_1(A)$. Then $[(a,u)]\ll [(b,v)]$ if and only if $[a]\ll[b]$ in $\Cu(A)$ and $\chi_{ab}([u])=[v]$ in $\K_1(\her(b))$.
\end{prop}

\begin{proof}
Suppose that $[(a,u)]\ll [(b,v)]$. A fortiori $[(a,u)]\leq[(b,v)]$, so $\chi_{ab}[u]=[v]$.
Now let $([c_n])_n$ be an increasing sequence in $\Cu(A)$ whose supremum $[c]$ satisfies $[c]\geq [b]$. Write $w:=\theta_{bc}(v)$ and consider $s:=[(c,w)]\in\Cu_1(A)$. By \autoref{cor:existencelift}, we know that there exists $n\in\N$ and a unitary element $w_n$ of $\her(c_n)^\sim$ such that $\chi_{c_nc}([w_n])=[w]$. Now define $s_k:=[c_{n+k},\theta_{c_nc_{n+k}}(w_n)]$. Then $(s_k)_k$ is an increasing sequence in $\Cu_1(A)$. By the description of suprema obtained in \autoref{lma:O1}, we know that $(s_k)_k$ admits $s$ as a supremum. Further, $s\geq [(b,v)]$ and since $[(a,u)]\ll [(b,v)]$, we deduce that there exists $k\in\N$ such that $[(a,u)]\leq s_k$ and hence that $[a]\leq [c_{n+k}]$. We conclude that $[a]\ll [b]$ in $\Cu(A)$.
                                                       
Conversely, let $[(a,u)],[(b,v)]\in\Cu_1(A)$ such that $[a]\ll [b]$ in $\Cu(A)$ and $\chi_{ab}[u]=[v]$ in $\K_1(\her(b))$. Let $([(c_n,w_n)])_n$ be an increasing sequence in $\Cu_1(A)$ that has a supremum in $\Cu_1(A)$, say $[(c,w)]$. Also suppose that $[(b,v)]\leq [(c,w)]$. First, by transitivity of $\lesssim_1$, observe that $\chi_{ac}([u])= \chi_{bc}\circ\chi_{ab}([u])= [w]$ in $\K_1(\her(c))$.
 
Arguing as in the proof of \cite[Lemma 4.3]{BPT06}, since $A$ has stable rank one, we can find a strictly decreasing sequence $(\epsilon_n)_n$ in $\R_+^*$ and unitary elements $(u_n)_n$ in $(A\otimes\mathcal{K})^\sim$ such that
\[
\her(c_1-\epsilon_1)_+\subseteq u_1(\her(c_2-\epsilon_2)_+)u_1^*\subseteq ...\subseteq u_n...u_1(\her(c_{n+1}-\epsilon_{n+1})_+)u_1^*...u_n^*\subseteq... 
\]
and such that $\sup\limits_n [(c_n-\epsilon_n)_+]=[c]$ in $\Cu(A)$. 
Hence, by \autoref{cor:existencelift} we can find $k\in\N$ and a unitary element $\tilde{w_k}$ of $(\her (c_k-\epsilon_k)_+)^\sim$ such that $\chi_{(c_k-\epsilon_k)_+c_k}[\tilde{w_k}]=[w_k]$ in $\K_1(\her(c_k))$.
Now, using the same argument as in the proof of \autoref{cor:existencelift}, we observe that 
\[
\AbGp-\lim\limits_{\longrightarrow}(\K_1(\her (c_n-\epsilon_n)_+),\chi_{(c_n-\epsilon_n)_+ (c_m-\epsilon_m)_+})\simeq (\K_1(\her(c)),\chi_{(c_n-\epsilon_n)_+c}).
\] 
On the other hand, since $[a]\ll [b] \leq \sup\limits_n [(c_n-\epsilon_n)_+] $, there exists $l\in\N$ such that $[a]\leq [(c_{l}-\epsilon_l)_+]$ in $\Cu(A)$. Without loss of generality,  $l\geq k$.
Using transitivity of $\lesssim_1$ again, we have that $\chi_{(c_l-\epsilon_l)_+c}([\tilde{w_l}])=\chi_{c_l c}\circ\chi_{(c_l-\epsilon_l)_+c_l} ([\tilde{w_l}])=[w]=\chi_{ac}([u])= \chi_{(c_l-\epsilon_l)_+c}\circ\chi_{a(c_l-\epsilon_l)_+}([u])$ in $\K_1(\her(c))$. Since we are in the category $\AbGp$, there exists some $l'\geq l$ such that $\chi_{(c_l-\epsilon_l)_+ (c_{l'}-\epsilon_{l'})_+}([\tilde{w_l}])= \chi_{(c_l-\epsilon_l)_+ (c_{l'}-\epsilon_{l'})_+}\circ\chi_{(ac_l-\epsilon_l)_+}([u])$. Composing with $\chi_{(c_{l'}-\epsilon_{l'})_+ c_{l'}}$ on both sides, we finally obtain that $[w_{l'}]=\chi_{ac_{l'}}[u]$ and hence $[(a,u)]\leq [(c_{l'},w_{l'})]$, which completes the proof.
\end{proof}

\begin{cor}
\label{cor:compactcu1}
Let $A$ be a $\CatCa$-algebra of stable rank one and let $[(a,u)]\in\Cu_1(A)$. Then $[(a,u)]$ is compact if and only if $[a]$ is compact in $\Cu(A)$.
\end{cor}

\begin{thm}
\label{prop:cuntzaxiom}
Let $A$ be a $\CatCa$-algebra of stable rank one. Then $(\Cu_1(A),\leq)$ satisfies axioms (O1), (O2), (O3), and (O4).
\end{thm}

\begin{proof}
(O1) follows from \autoref{lma:O1}.

(O2): Let $s:=[(a,u)]\in\Cu_1(A)$. We want to write $s$ as the supremum of a $\ll$-increasing sequence in $\Cu_1(A)$. By (O2), we can find a $\ll$-increasing sequence $([a_n])_n$ in $\Cu(A)$ such that $\sup\limits_n[a_n]=[a]$. Let $a_n$ be any representative of $[a_n]$ in $(A\otimes\mathcal{K})_+$. Using \autoref{cor:existencelift}, we know that we can find a unitary element $u_n$ of $\her(a_n)^\sim$ for some $n\in\N$ such that $[(a_n,u_n)]\leq [(a,u)]$. Now we consider $s_k:=[(a_{n+k},\theta_{a_na_{n+k}}(u_n))]$, for any $k\in\N$. Then, by \autoref{prop:llCU1} we deduce that $(s_k)_k$ is a $\ll$-increasing sequence in $\Cu_1(A)$. By the description of suprema obtained in \autoref{lma:O1}, $\sup\limits_k s_k=s$.

(O3): Let $[(a_1,u_1)]\ll [(b_1,v_1)]$ and $[(a_2,u_2)]\ll [(b_2,v_2)]$. We already know that $[(a_1,u_1)]+[(a_2,u_2)]\leq [(b_1,v_1)]+[(b_2,v_2)]$ and that $[a_1]+[a_2]\ll [b_1]+[b_2]$ in $\Cu(A)$. The conclusion follows from \autoref{prop:llCU1}. 

(O4): Let $([(a_n,u_n)])_{n\in\N}$ and $([(b_n,v_n)])_{n\in\N}$  be two increasing sequences in $\Cu_1(A)$. Let $[(a,u)]:= \sup\limits_{n\in\N}[(a_n,u_n)]$ and $[(b,v)]:= \sup\limits_{n\in\N}[(b_n,v_n)]$. Now we define $[(c_n,w_n)] := [(a_n,u_n)]+[(b_n,v_n)]$ for any $n\in\N$. Since $[c_n]=[a_n]+[b_n]$ in $\Cu(A)$ and $\Cu(A)$ satisfies (O4), we have $\sup\limits_{n\in\N} [c_n]= [a\oplus b ]$. Also, we know that $\chi_{a_na}([u_n])=[u]$ and $\chi_{b_nb}([v_n])=[v]$, and hence we obtain $\chi_{c_nc}(u_n\oplus v_n)=u\oplus v$. We conclude that $\sup$ and addition are compatible in $\Cu_1(A)$, using \autoref{lma:O1}.
\end{proof}

\subsection{The \texorpdfstring{$\Cu_1$}{Cu1}-semigroup as a functor}

From now on, let us write $\CatCasr$ to denote the category of $\CatCa$-algebras of stable rank one. Also, we denote by $\CatoM$ the category of ordered monoids, in contrast to the category of positively ordered monoids, that we write $\CatPoM$. Finally, the category of monoids is denoted by $\CatM$.
We have just proved that $\Cu_1(A)$ is a semigroup satisfying the axioms (O1)-(O4). The aim is to define a functor $\Cu_1$ from the category $\CatCasr$ to a suitable category of semigroups as was done for the $\Cu$-semigroup; see \cite[Chapter 3]{APT14}, \cite{CEI08}. Since $\Cu_1(A)$ is usually not positively ordered, we need to adjust the definition of the codomain category. In the sequel, we show that $\Cu_1:\CatCasr\longrightarrow \Cu^\sim$ is a well-defined functor that is continuous. 

\begin{dfn}
\label{dfn:Cusimcat}
The \emph{unitary Cuntz category}, written $\Cu^\sim$ is the subcategory of $\CatoM$ whose objects are ordered monoids satisfying the axioms (O1)-(O4) and such that $0\ll 0$. Morphisms in $\Cu^\sim$ are $\CatoM$-morphisms that respect suprema of increasing sequences and the compact-containment relation.
\end{dfn}

\begin{dfn}
\label{positivecone}
Let $M\in\CatoM$ and let $S\in\Cu^\sim$. We define their \emph{positive cones}, that we write $M_+$ and $S_+$ respectively, as the subset of positive elements. Observe that $M_+\in\CatPoM$ and $S_+\in\Cu$.
\end{dfn}

\begin{lma}
\label{prop:nu+}
The category $\Cu$ (respectively $\CatPoM$) is a coreflective subcategory of $\Cu^\sim$ (respectively $\CatoM$). More precisely, the assignment $S\longrightarrow S_+$ defines a coreflector $\nu_+:\Cu^\sim\longrightarrow \Cu$.
\end{lma}

\begin{proof}
Since $\Cu^{\sim}$-morphisms respect $\leq$, we deduce that $\nu_+$ is a well-defined functor.
Moreover, one can check that $\Hom_{\Cu^{\sim}}(i(S),T)\simeq \Hom_{\Cu}(S,\nu_+(T))$ for any $S\in \Cu$ and $T\in\Cu^{\sim}$. 
We get that the inclusion functor $i:\Cu\hooklongrightarrow \Cu^\sim$ is left adjoint to $\nu_+$, which implies that $\Cu$ is a full (obviously faithful) coreflective subcategory of $\Cu^\sim$.
\end{proof}

\begin{prop}
\label{prop:cu1functor}
Let $\varphi: A \longrightarrow B$ a $^*$-homomorphism between $\CatCa$-algebras $A,B$ of stable rank one. We denote by $\varphi^\sim$ the unitized morphism between $(A\otimes\mathcal{K})^\sim\longrightarrow (B\otimes\mathcal{K})^\sim$. Then:
\vspace{-0,1cm}\[
	\begin{array}{ll}
		\Cu_{1}(\varphi): \Cu_1(A) \longrightarrow \Cu_1(B)\\
		\hspace{1,35cm} [(a,u)] \longmapsto [(\varphi(a),\varphi^{\sim}(u))]
	\end{array}
\vspace{-0,1cm}\] 
is a $\Cu^{\sim}$-morphism.
\end{prop}

\begin{proof}
Let $a\in A\otimes\mathcal{K}$. The restriction $\varphi_{|\her(a)}:\her(a)\longrightarrow \her (\varphi(a))$ of $\varphi$ gives us the following commutative square:
\vspace{-0,5cm}\[
\xymatrix@C+1pc@R+0pc{
\her(a)\ar[d]_{}\ar[r]^{\hspace{-0,5cm} \varphi} & \her(\varphi(a))\ar[d]^{}  \\
\her(a)^\sim\ar[r]_{\hspace{-0,5cm} \varphi^{\sim}} & (\her (\varphi(a)))^\sim
} 
\]
Hence, $\varphi^{\sim}(u)$ is a unitary element of $(\her (\varphi(a)))^\sim$ and we deduce that $[(\varphi(a), \varphi^{\sim}(u))]\in \Cu_1(B)$. Let us check it does not depend on the representative $(a,u)$ chosen. Let $[(a,u)],[(b,v)]\in \Cu_1(A)$ such that $[(a,u)]\leq[(b,v)]$. Then we get $a\lesssim_{\Cu}b$ in $A\otimes\mathcal{K}$. Since $\varphi$ is a $^*$-homo\-morphism, we deduce that $\varphi(a)\lesssim_{\Cu} \varphi(b)$ in $B\otimes\mathcal{K}$. Further, if $\alpha$ is a partial isometry of $(A\otimes\mathcal{K})^{**}$  that realizes one of our standard morphisms $\theta_{ab,\alpha}$ (see \autoref{rmk:canoinj}) between $\her(a)$ and $\her(b)$, then $\varphi^{**}(\alpha)$ is a partial isometry of $(B\otimes\mathcal{K})^{**}$ that realizes $\theta_{\varphi(a)\varphi(b),\varphi^{**}(\alpha)}$ between $\her (\varphi(a))$ and $\varphi(b)$, since $\varphi^{**}$ is a $^*$-homomorphism. We get that the following diagram is commutative: 
\vspace{-0,45cm}\[
\vspace{-0cm}\hspace{1,5cm}\xymatrix@C+1pc@R+0pc{
\her(a)^\sim\ar[r]^{\theta_{ab,\alpha}}\ar[d]_{\varphi^{\sim}} & \her(b)^\sim\ar[d]^{\varphi^{\sim}}  \\
(\her (\varphi(a)))^\sim\ar[r]_{{\theta_{\varphi(a)\varphi(b),\varphi^{**}(\alpha)}}} & (\her (\varphi(b)))^\sim
} 
\vspace{-0,15cm}\]
from which we deduce that $\theta_{\varphi(a)\varphi(b)}(\varphi^{\sim}(u))\sim \varphi^{\sim}(v)$ and thus $[(\varphi(a), \varphi^{\sim}(u))]\leq[(\varphi(b), \varphi^{\sim}(v))]$. 
So $\Cu_{1}(\varphi)$ is indeed well-defined, respects $\leq$ and it is easy to check that $\Cu_1(\varphi)$ also respects addition. We conclude that $\Cu_{1}(\varphi)$ is a $\CatoM$-morphism. 
By \autoref{prop:llCU1}, $\Cu_{1}(\varphi)$ preserves the compact containment relation. Finally, we leave to the reader to check that $\Cu_{1}(\varphi)$ preserves suprema of increasing sequences.
\end{proof}

\begin{cor}
The assignment $A\longmapsto \Cu_1(A)$ from $\CatCasr$ to $\Cu^\sim$ is a functor.
\end{cor}

It has been shown that the functor $\Cu$ from the category of $\CatCa$-algebras to $\Cu$ is continuous (\cite[Corollary 3.2.9]{APT14}), generalizing the result of \cite[Theorem 2]{CEI08} that established sequential continuity.
We shall expect a similar result for the functor $\Cu_1$. In the sequel, we shall prove that $\Cu_1:\CatCasr\longrightarrow \Cu^\sim$ is a continuous functor, using a process analogous to that in \cite[Chapter 2 and 3]{APT14} and \cite[Section 2.2]{APT19} for the Cuntz semigroup. 

To do so, we are going to consider a pre-completed version of $\Cu_1$, that we will denote by $\W_1$, to then extend the result to $\Cu_1$ using Category Theory techniques. 
We first introduce an analogous category to $\CatW$ defined in \cite[Definition 2.5]{APT19} that we shall call $\CatW^{\sim}$.
The main difference of our context lies in the fact that binary/auxiliary relations considered need not be positive, and similarly, the underlying ordered monoids involved are not necessarily positively ordered. Still, most of the proofs from \cite{APT14} and \cite{APT19} remain valid. (We give additional details when needed.) 

\subsection{The category \texorpdfstring{$\CatW^\sim$}{W-tilde}}
\label{dfn:PreWW}
Let $S\in\CatM$ and consider a transitive binary relation $\prec$ on $S$. (Again, we do not require $\prec$ to be positive, in the sense that there may exist $s\in S$ such that $0\nprec s$.) For any $s\in S$ we denote $s_{\prec}:=\{s'\in S\mid s'\prec s\}$. Let us recall the $\W$-axioms from \cite[Definition 2.2]{APT19}:

(W1): For any $s\in S$, there exists a $\prec$-increasing sequence $(s_k)_k$ in $s_\prec$ such that for any $s'\in s_\prec$, there exists some $k$ such that $s'\prec s_k$. 

(W3): Addition and $\prec$ are compatible. 

(W4): For any $s,t,x\in S$ such that $x\prec s+t$, we can find $s',t'\in S$ such that $s'\prec s, t'\prec t$ and $x\prec s'+t'$.\\
A \emph{$\CatW^\sim$-semigroup} is a pair $(S,\prec)$, where $S\in\CatM$ and $\prec$ is a transitive binary relation (not necessarily positive) on $S$ such that $(S,\prec)$ satisfies axioms (W1)-(W3)-(W4) and such that $0\prec 0$.\\
A \emph{$\W^\sim$-morphism} between any two $S,T\in \CatW^\sim$ is a $\CatM$-morphism $g:S\longrightarrow T$ that respects the transitive binary relation and satisfies the following \emph{$\CatW^\sim$-continuity axiom}:

(M): For any $s\in S$ and $t\in T$ such that $t\prec g(s)$, there exists $s'\in s_\prec$ such that $t\prec g(s')$.

The category $\CatW^{\sim}$ has inductive limits. More precisely, let $(S_i,\varphi_{ij})_{i \in I}$ be an inductive system in $\CatW^{\sim}$ and let  $S:=\CatM-\lim\limits_{\longrightarrow}(S_i,\varphi_{ij})$. Then $(S,\prec)\simeq \CatW^\sim-\lim\limits_{\longrightarrow}(S_i,\varphi_{ij})$, where $\prec$ is the following transitive binary relation on $S$: $s\prec t$ in $S$ if $\varphi_{ik}(s_i) \prec \varphi_{jk}(t_j)$, where $s_i\in S_i$, $t_j\in S_j$ are representatives of $s,t$ respectively and $k\geq i,j$. 

Now that we have a well-defined categorical setup, we define a pre-completed version of $\Cu_1$ and show that it is continuous. More precisely, we build a functor from the category $C^{*}_{loc}$ of local $\CatCa$-algebras to the category $\CatW^{\sim}$, termed $\CatW_1$. 
See \cite[\S 2.2]{APT14} for more details.

\subsection{Local \texorpdfstring{$\CatCa$}{C*}-algebras}
\label{prg:localcat}
A \emph{local $\CatCa$-algebra} $A$ is an upward-directed union of $\CatCa$-algebras. That is, $A=\underset{i}\cup A_i$ where $\{A_i\}_i$ is a family of complete $^*$-invariant subalgebras such that for any $i,j$, there exists $k\geq i,j$ such that $A_i\cup A_j\subseteq A_k$. 

If $A$ is a local $\CatCa$-algebra, then so is $M_k(A)$ for any $k\in\N$. In fact, $M_k(A)$ sits as upper-left corner inside $M_{k'}(A)$ for any $k'\geq k$ and we can picture any $M_k(A)$ as a corner of $M_{\infty}(A):=\underset{k}\cup M_k(A)$, which is again a local $\CatCa$-algebra.
Observe that the completion of a local $\CatCa$-algebra $A$, that we write $\overline{A}$, is a $\CatCa$-algebra. In particular, we have $\overline{M_k(A)}\simeq M_k(\overline{A})$ for any $k\in\N$ and $\overline{M_{\infty}(A)}\simeq \overline{A}\otimes\mathcal{K}$. Further $A$ is closed under functional calculus.
Moreover, for any local $\CatCa$-algebra $A:=\underset{i}\cup A_i$, if each $A_i$ has stable rank one, then by \cite[Theorem 5.1]{R83}, we get that $\overline{A}$ has stable rank one. We may abuse the language and say that $A$ has stable rank one.

We now consider $\CatCa_{loc}$, the category whose objects are local $\CatCa$-algebras and morphisms are $^*$-homo\-morphisms. Obviously, $\CatCa$ is a full subcategory of $\CatCa_{loc}$. In fact, $\CatCa$ is a reflective subcategory of $\CatCa_{loc}$ and the assignment $A\longmapsto \overline{A}$ defines a reflector from $\CatCa_{loc}$ to $\CatCa$ that we denote by $\gamma$. As for $\CatCa$-algebras, we denote $\CatCa_{loc,sr1}$ the full subcategory of $\CatCa_{loc}$ consisting of local $\CatCa$-algebras whose completion have stable rank one.

Finally, let $(A_i,\varphi_{ij})_{i \in I}$ be an inductive system in $\CatCa_{loc}$. As in \cite[\S 2.2.8]{APT14}, we consider the algebraic inductive limit $A_{alg}:= \bigsqcup\limits_{i \in I}A_i/\!\!\sim$ with the pre-norm: $\Vert x\Vert:=\inf\limits_j\{\Vert \varphi_{ij}(x)\Vert \})$, for $x\in A_i$ and we define:
\vspace{-0,2cm}\[
\CatCa_{loc}-\lim\limits_{\longrightarrow}(A_i,\varphi_{ij}):=(A_{alg}/N,\Vert\,\Vert)
\]
where $N:=\{ a\in A_{alg} \mid \Vert a\Vert =0 \}$.
Besides, $\varphi_{ij}$ induces a $^*$-homomorphism that we also write $\varphi_{ij}: M_\infty(A_i)\longrightarrow M_\infty(A_j)$ and we have $\CatCa_{loc}-\lim\limits_{\longrightarrow}(M_\infty(A_i),\varphi_{ij}^{\sim})\simeq M_\infty(\CatCa_{loc}-\lim\limits_{\longrightarrow}(A_i,\varphi_{ij}))$. See \cite[\S 2.2.8]{APT14}.

\subsection{The precompleted unitary Cuntz semigroup}

We briefly recall the definition of the precompleted Cuntz semigroup $W(A)$ of a $\CatCa$-algebra $A$ and we refer the reader to \cite[\S 2.2]{APT14} for details. In fact, we give an equivalent definition that can be found in \cite[Remark 3.2.4]{APT14}; see also \cite[Lemma 3.2.7]{APT14}.

Let $A\in \CatCa_{loc}$. We define $\CatW(A):=\{[a]\in\Cu(\overline{A})\mid a\in M_{\infty}(A)_+\}$. Obviously, $(\CatW(A),+)\in\CatM$ as a submonoid of $\Cu(\overline{A})$. Given $[a], [b]\in\CatW(A)$, we write $[a]\prec[b]$ if $a\lesssim_{\Cu} (b-\epsilon)_{+}$ in  $M_{\infty}(A)_+$ for some $\epsilon>0$.
This defines a (positive) transitive binary relation on $\CatW(A)$, hence we have that $(\W(A),\prec)\in\CatW$. (See \cite[Proposition 2.2.5]{APT14} and \cite[Section 2.2]{APT19}.)

\begin{lma}
\label{lma:herloc}
Let $A\in\CatCa_{loc}$ and let $B:=\overline{A}$ be its completion in $\CatCa$. Then, for any $a\in A_+$ we have $\overline{aAa}=\overline{aBa}$.
\end{lma}

\begin{proof}
The inclusion $\subseteq$ is trivial. Now let $x\in\overline{aBa}$. Then there exists a sequence $(b_k)_k$ in $B$ such that $x=\lim\limits_k ab_ka$. Furthermore, for any $k\in\N$, there exists a sequence $(a_{k,i})_i$ in $A$ such that $b_k=\lim\limits_i a_{k,i}$. We deduce that $x= \lim\limits_k a(\lim\limits_i a_{k,i} )a = \lim\limits_k\lim\limits_i (aa_{k,i}a)$. Thus $x\in\overline{aAa}$.
\end{proof}

\begin{dfn}
\label{dfn:herloc}
Let $A\in\CatCa_{loc}$ and let $B:=\overline{A}$ be its completion as a $\CatCa$-algebra. For $a\in A_+$, we define the \emph{hereditary subalgebra generated by $a$} as $\her(a):= \overline{aBa}$.
\end{dfn}

We have now all the tools to define a precompleted version of $\Cu_1$ that we will denote by $\CatW_1(A)$, as a submonoid of $\Cu_1(\overline{A})$.

\begin{dfn}
\label{dfn:auxtild}
Let $A\in \CatCa_{loc,sr1}$. We define $\CatW_1(A):=\{[(a,u)]\in\Cu_1(\overline{A})\mid a\in M_{\infty}(A)_+\}$. Obviously, $(\CatW_1(A),+)\in\CatM$ as a submonoid of $\Cu_1(\overline{A})$. We now equip $\CatW_1(A)$ with the following binary relation: for any two $[(a,u)], [(b,v)]$ in $\CatW_1(A)$, we say $[(a,u)]\prec [(b,v)]$ if: 
\[
\hspace{0,1cm}\left\{
    \begin{array}{ll}
    		\,a\lesssim_{\Cu} (b-\epsilon)_{+} \text{ in } M_{\infty}(A)_+ \text{ for some $\epsilon>0$}.\\
        \,[\theta_{ab}(u)] = [v] \text{ in } 
\K_1(\her(b)^\sim).
    \end{array}
\right.
\]
\end{dfn}

\begin{prop}
\label{cor:existenceliftloc}
Let $A\in\CatCa_{loc,sr1}$. Let $a\in A_+$ and let $(a_n)_n$ be a sequence in $A_+$ such that $([a_n])_n$ is a $\prec$-increasing cofinal sequence in $[a]_\prec$ (obtained from $(\CatW 1)$ applied to $[a]$.)

For any unitary element $u\in \her(a)^\sim$, there exists $n\in\N$ and a unitary element $u_n$ in $\her(a_n)^\sim$ such that $[(a_n,u_n)]\prec [(a,u)]$ in $\W_1(A)$.
\end{prop}

\begin{proof}
Combine the fact that $\overline{A}$ has stable rank one, with \autoref{dfn:herloc} and the result follows from 
\autoref{cor:existencelift}.
\end{proof}

\begin{prop}(cf \cite[Proposition 2.2.5]{APT14}).
Let $A\in\CatCa_{loc,sr1}$. The relation defined in \autoref{dfn:auxtild} is a transitive binary relation and $(\CatW_1(A),\prec)$ satisfies axioms ($\CatW 1$), ($\CatW 3$) and ($\CatW 4$). That is, $(\CatW_1(A),\prec) \in\CatW^\sim$. We may omit the reference to $\prec$ and simply write $\CatW_1(A)\in\CatW^\sim$.
\end{prop}

\begin{proof}
Let us check that $\prec$ is transitive. If $[(a,u)]\prec [(b,v)]\prec[(c,w)]$, then we have $\chi_{ac}([u])=[w]$ and we also know that $a\lesssim_{\Cu} (b-\epsilon)_+$. Using \cite[Proposition 2.4]{R92}, we can find some $\delta>0$ such that  $a\lesssim_{\Cu} (c-\delta)_+$, since $a\lesssim_{\Cu}b\lesssim_{\Cu}(c-\epsilon')_+$ for some $\epsilon'>0$. We conclude that $[(a,u)]\prec[(c,w)]$.

If $[(b,v)]\prec [(a,u)]$, then, by \autoref{prop:llCU1}, we have $[(b,v)]\ll [(a,u)]$ in $\Cu_1(A)$ and thus $[(b,v)]\prec [((a-1/n)+,u_n)]$ for some $n\in\N$. Hence (W1) holds. To check (W3) and (W4) is routine.
\end{proof}

\begin{prop}
\label{prop:w1phi}
Let $\varphi: A\longrightarrow B$ be a $^*$-homomorphism between $A,B\in\CatCa_{loc,sr1}$, and denote by $\varphi$ its extension to $M_{\infty}(A)$. We write $\overline{\varphi}:=\gamma(\varphi)$ and $\overline{\varphi}^\sim:\overline{M_{\infty}(A)}^\sim\longrightarrow \overline{M_{\infty}(B)}^\sim$ its unitization. Then the map:
\[
	\begin{array}{ll}
			\CatW_{1}(\varphi): \CatW_1(A) \longrightarrow \CatW_1(B)\\
			\hspace{1,2cm} [(a,u)] \longmapsto [(\varphi(a),\overline{\varphi}^{\sim}(u))]
	\end{array}
\]
is a $\CatW^{\sim}$-morphism. 
\end{prop}

\begin{proof}
Using the same argument as in \autoref{prop:cu1functor}, we easily deduce that $\CatW_{1}(\varphi)$ is a $\CatM$-morphism that respects $\prec$.
Further, we have to check that $\CatW_1(\varphi)$ satisfies the $\CatW^\sim$-continuity axiom (see \autoref{dfn:PreWW}). Let us write $f:=\CatW_{1}(\varphi)$. Let $x:=[(a,u)]\in \CatW_1(A)$ and $y:=[(b,v)]\in \CatW_1(B)$ such that $y\prec f(x)$. We have to find $x'\in\W_1(A)$ such that $x'\prec x$ and $y\prec f(x')$.

Observe that $([(a-1/n)_+])_n$ is one of the $\prec$-increasing sequences obtained from axiom $(\CatW1)$ applied to $[a]$ in $\W(A)$. Thus, by \autoref{cor:existenceliftloc}, we can find some $n\in\N$ and a unitary element $u_n\in\her((a-1/n)_+)^\sim$ such that $[((a-1/n)_+,u_n)]\prec [(a,u)]$ in $\CatW_1(A)$. Similarly, $([(\varphi(a)-1/k)_+])_k$ is one of the $\prec$-increasing sequences obtained from $(\CatW1)$ applied to $[\varphi(a)]$ in $\W(B)$. Therefore, there exists $k\in\N$ such that $[b]\prec [(\varphi(a)-1/k)_+]$ in $\W(B)$. We deduce that there exists $m\in\N$ large enough ($m\geq k,n)$ such that:
\[
\left\{
    \begin{array}{ll}
    		\,[b]\prec [(\varphi(a)-1/m)_+] \text{ in } \CatW(B).\\
        \,[\theta_{b\varphi(a)}(v)] = [\theta_{(\varphi(a)-1/n)_+\varphi(a)}(\overline{\varphi}^{\sim} (u_n))] \text{ in } 
\K_1(\her (\varphi(a))).
    \end{array}
\right.
\]
By transitivity of $\lesssim_1$, we obtain: 
\[
[\theta_{(\varphi(a)-1/m)_+\varphi(a)}\circ \theta_{b(\varphi(a)-1/m)_+}(v)] = [\theta_{(\varphi(a)-1/m)_+\varphi(a)}\circ\theta_{(\varphi(a)-1/n)_+ (\varphi(a)-1/m)_+}(\overline{\varphi}^{\sim} (u_n))] \text{ in } 
\K_1(\her (\varphi(a))).
\]
Finally, since $\AbGp-\lim\limits_{\longrightarrow}(\K_1(\her (\varphi(a)-1/m)_+),\chi_{(\varphi(a)-1/n)_+(\varphi(a)-1/m)_+})\simeq (\K_1(\her(a)),\chi_{(\varphi(a)-1/m)_+\varphi(a)})$, we conclude that there exists $l\geq m$ such that:
\[
\left\{
    \begin{array}{ll}
    		\,[b]\prec [(\varphi(a)-1/l)_+] \text{ in } \CatW(B).\\
        \,[\theta_{b(\varphi(a)-1/l)_+}(v)] = [\theta_{(\varphi(a)-1/n)_+ (\varphi(a)-1/l)_+}(\overline{\varphi}^{\sim}(u_n))] \text{ in } 
\K_1(\her (\varphi(a)-1/l)_+).
    \end{array}
\right.
\]
Write $x':=[((a-1/l)_+,\theta_{(a-1/n)_+(a-1/l)_+}(u_n))]$. Then we already know that $x'\prec x$ in $\W_1(A)$ and the above exactly states that $y\prec f(x')$ in $\CatW_1(B)$.
\end{proof}

\begin{cor}
The assignment $A\longmapsto \W_1(A)$ from $\CatCa_{loc,sr1}$ to $\W^\sim$ is a functor.
\end{cor}

\begin{thm}
The functor $\CatW_1:C^*_{loc,sr1}\longrightarrow \CatW^{\sim}$ is continuous.
\end{thm}

\begin{proof}
This proof is an adapted version of \cite[Theorem 2.2.9]{APT14} and \cite[Theorem 2.9]{APT19}.
Let $(A_i,\varphi_{ij})_{i \in I}$ be an inductive system in $C^*_{loc,sr1}$ and let $(A_{alg}/N,\varphi_{i\infty})$ be its inductive limit. Without loss of generality, we can suppose that each $A_i\simeq M_\infty(A_i)$; see \autoref{prg:localcat}. Thus, we may suppose that each element of $\CatW(A_i)$ is realized by a positive element of $A_i$. 

Let $\sigma_{ij}:=\CatW_1(\varphi_{ij})$ and consider the inductive system $(\CatW_1(A_i),\psi_{ij})_{i \in I}$ in $\CatW^{\sim}$. We denote by $(S,\sigma_{i\infty})$ its inductive limit in $\CatW^{\sim}$. Observe that $(\CatW_1(A_{alg}/N),\CatW_1(\varphi_{i\infty}))$ is a co-cone for the inductive system. Hence from universal properties, we deduce that there exists a unique $\CatW^\sim$-morphism $w_1:S\longrightarrow \CatW_1(A_{alg}/N)$ such that for all $i,j\in I$ with $i\leq j$, the following diagram commutes: 
\[
\xymatrix{
\CatW_1(A_i)\ar[dd]_{\sigma_{ij}=\CatW_1(\varphi_{ij})}\ar[rd]^{\sigma_{i\infty}}\ar@/^2pc/[rrd]^{\CatW_1(\varphi_{i\infty})} &  &\\
& S\ar@{.>}[r]_{\hspace{-0,9cm}\exists!\, w_1}&  \CatW_1(A_{alg}/N) \\
\CatW_1(A_j)\ar[ru]_{\sigma_{j\infty}}\ar@/_2pc/[rru]_{\CatW_1(\varphi_{j\infty})}  & &
} 
\]
To complete the proof, let us show that $w_1$ is a $\CatW^{\sim}$-isomorphism. 
First, we start to show that $w_1$ is surjective. Let $[(a,u)]\in \CatW_1(A_{alg}/N)$. Since $a\in A_{alg}/N$, we know that there exists $a_k\in (A_k)_+$ such that $\varphi_{k\infty}(a_k)=a$. Also, $u$ is a unitary element of $\her(a)^\sim=\overline{\varphi_{k\infty}(a_k) (A_{alg}/N) \varphi_{k\infty}(a_k)}^\sim$. Now, observe that $\CatCa-\lim\limits_{\longrightarrow j>k}(\her \varphi_{kj}(a_k),\varphi_{jl})\simeq (\her(a),\varphi_{j\infty})$. Hence for any $\epsilon>0$, there exists $j\geq k$ and a unitary element $u_j$ of $\her \varphi_{kj}(a_k)^\sim$ such that $\Vert u - \overline{\varphi}^\sim_{j\infty}(u_j)\Vert < \epsilon$. In particular, for $\epsilon<2$, we obtain a unitary element $u_j$ of $\her \varphi_{kj}(a_k)^\sim$ such that $[u]=[\overline{\varphi}^\sim_{j\infty}(u_j)]$ in $\K_1(\her(a))$. 
We compute that $\W_1(\varphi_{j\infty})([(\varphi_{kj}(a_k),u_j)])=[(\varphi_{k\infty}(a_k),\overline{\varphi_{j\infty}}^\sim(u_j))]=[(a,u)]$.

Thus, by the commutativity of the diagram above we obtain 
\vspace{-0,1cm}\[w_1\circ\sigma_{j\infty}([(\varphi_{kj}(a_k),u_j)])= \W_1(\varphi_{j\infty})([(\varphi_{kj}(a_k),u_j)]) =[(a,u)]
\vspace{-0,1cm}\] as desired. We conclude that $w_1$ is surjective. 

Finally, let us show that $w_1$ is injective. Let $s,t\in S$ such that $w_1(s) = w_1(t)$. Since the inductive limit $S$ is algebraic, there exists some $k\in\N$ and $s_k,t_k$ in $W_1(A_k)$ such that $\sigma_{k\infty}(s_k)=s$ and $\sigma_{k\infty}(t_k)=t$. 
 Now choose $a,b\in (A_k)_+$ and unitary elements $u, v$ in the respective hereditary subalgebras such that $s_k=[(a,u)] $ and $t_k=[(b,v)]$. We know that $w_1(s)= w_1(t)$ and using the commutativity of the above diagram, we deduce that
\[
\hspace{0,1cm}\left\{
    \begin{array}{ll}
    		\, [\varphi_{k\infty}(a)]= [\varphi_{k\infty}(b)] \text{ in } \CatW(A_{alg}/N).\\
        \,[\theta_{\varphi_{k\infty}(a) \varphi_{k\infty}(b)}(\overline{\varphi_{k\infty}}^\sim(u))]= [\overline{\varphi_{k\infty}}^\sim(v)] \text{ in } 
\K_1(\her \varphi_{k\infty}(b)).
    \end{array}
\right.
\]
Again, since the inductive limits are algebraic, we conclude that there exists $l\geq k$ such that:
\[
\hspace{0,1cm}\left\{
    \begin{array}{ll}
    		\,[\varphi_{kl}(a)]=[\varphi_{kl}(b)] \text{ in } \CatW(A_l).\\
        \,[\theta_{\varphi_{kl}(a) \varphi_{kl}(b)}(\overline{\varphi_{kl}}^\sim(u))]= [\overline{\varphi_{kl}}^\sim(v)] \text{ in } 
\K_1(\her \varphi_{kl}(b)).
    \end{array}
\right.
\]
We conclude that $\sigma_{kl}(s_k)=\sigma_{kl}(t_k)$ for some $l\geq k$. Thus $s=t$, which ends the proof.
\end{proof}

\subsection{Continuity of the functor \texorpdfstring{$\Cu_1$}{Cu1}}
\label{prg:identification}
We now have all the tools to conclude that $\Cu_1:\Cu^\sim\longrightarrow \W^\sim$ is a continuous functor, using the same techniques as in \cite[Chapter 3]{APT14}. First of all, using a similar argument as in \cite[Proposition 3.1.6]{APT14}, we easily deduce the following:\\
Let $(S,\prec)$ be a $\CatW^{\sim}$-semigroup. Then there exists a $\Cu^\sim$-semigroup $\gamma^{\sim}(S)$ together with a $\CatW^\sim$-morphism $\alpha_{S}:S\longrightarrow \gamma^\sim(S)$ satisfying the following conditions:

(i) The morphism $\alpha_{S}$ is an \textquoteleft $\prec$-embedding\textquoteright\, in the sense that $s'\prec s$ whenever $\alpha(s')\ll \alpha(s)$.

(ii) The morphism $\alpha_{S}$ has a \textquoteleft dense image\textquoteright\, in the sense that for any two $t',t\in\gamma^\sim(S)$ such that $t'\ll t$ there exists $s\in S$ such that $t'\leq \alpha(s)\leq t$.

Note that the construction of such a \emph{completion} is similar in every way except that we do not impose the transitive binary relation on $S$ to be positive. This implies that the ordered monoid obtained respects the axioms (O1)-(O4) but need not be positively ordered, whence $\gamma^{\sim}(S)$ belongs to $\Cu^\sim$ instead of $\Cu$. 
Again, arguing as in \cite[Theorem 3.1.8]{APT14}, we deduce that $\Cu^{\sim}$ is a (full) reflective subcategory of $\CatW^{\sim}$ with reflector $\gamma^\sim$. In particular, $\Cu^{\sim}$ has inductive limits. 
Finally, observe that for any $A\in\CatCasr$, the compact-containment relation on $\Cu_1(A)$ and the $\prec$ relation on $\CatW_1(A\otimes\mathcal{K})$ agree; see \cite[Remark 3.2.4]{APT14}. Thus,
we have that $\Cu_1(A)=\CatW_1(A\otimes\mathcal{K})$ as $\Cu^{\sim}$-semigroups.

\begin{thm}
There exists a natural isomorphism $\gamma^{\sim}\circ \CatW_1\simeq\Cu_1\circ\gamma$, where $\gamma$ is the reflector from $\CatCa_{loc,sr1}$ to $\CatCasr$ defined in \autoref{prg:localcat}. In particular, for any $\CatCa$-algebra $A$ of stable rank one, there is a (natural) $\Cu^\sim$-isomorphism between $\Cu_1(A)\simeq \gamma^\sim(\W_1(A))$.
\end{thm}

\begin{proof}
The aim of the proof is to show that $(\Cu_1(\gamma(A)),W_1(i))$ is a $\Cu^\sim$-completion of $W_1(A)$ for any $A\in \CatCa_{loc,sr1}$, where $W_1(i)$ is built as follows: 

Let $A\in\CatCa_{loc,sr1}$, write $B:=M_{\infty}(A)\in \CatCa_{loc,sr1}$. Consider the canonical inclusion $i:B\hooklongrightarrow \overline{B}\simeq\overline{A}\otimes\mathcal{K}$. Then $i$ induces a $\CatW^\sim$-morphism $W_1(i):\CatW_1(B)\longrightarrow \CatW_1(\overline{B})$. On the other hand, we know that $\CatW_1(B)=\CatW_1(A)$ and that $\CatW_1(\overline{B})\simeq\Cu_1(\overline{A})$. Thus, we obtain a $\CatW^\sim$-morphism $W(i):\W_1(A)\longrightarrow \Cu_1(\overline{A})$ (we use the same notation).
By the argument in \cite[Theorem 3.1.8]{APT14}, we only have to check that $W_1(i)$ is a $\prec$-embedding and that it has a dense image. 

Let $s,s'\in\W_1(A)$ such that $\W_1(i)(s')\ll \W_1(i)(s')$. We deduce that $\W_1(i)(s')\prec\W_1(i)(s')$. Also, observe that $\W_1(i)$ is in fact an order embedding (even more, it is the canonical injection). Thus, we conclude that $s\prec s'$ and hence $\W_1(i)$ is an \textquoteleft $\prec$-embedding\textquoteright.

Let $t,t'\in\Cu_1(\gamma (A))$ such that $t' \ll t$. Now pick $a,a'\in (\gamma(A)\otimes\mathcal{K})_+$ and unitary elements $u,u'$ in the respective hereditary subalgebras of $a,a'$, such that $t:=[(a,u)]$ and $t':=[(a',u')]$. Then, we know that $[a']\ll[a]$ in $\Cu(\overline{A})$ and that $\chi_{a'a}([u'])=[u]$. Using the argument in \cite[Lemma 3.2.7]{APT14}, there exists $b \in M_{\infty}(A)_+$ such that $[a']\leq [b]\leq [a]$ in $\Cu(\overline{A})$. Now consider $s:=[(b,\theta_{a'b}(u))]\in \W_1(A)$ and we get that $t'\leq \W_1(i)(s)\leq t$ in $\Cu_1(\overline{A})$. It follows that $\W_1(i)$ has a \textquoteleft dense image\textquoteright\, and hence that $(\W_1(i), \Cu_1(\gamma (A)))$ is a $\Cu^{\sim}$-completion of $\CatW_1(A)$.
\end{proof}

\begin{cor}
\label{cor:cu1ctinuous}
The functor $\Cu_1: \CatCasr \longrightarrow \Cu^{\sim}$ is continuous. More precisely, given an inductive system $(A_i,\phi_{ij})_{i\in I}$ in $\CatCasr$, then:
\[
\Cu^\sim-\lim\limits_{\longrightarrow}(\Cu_1(A_i),\Cu_1(\phi_{ij}))\simeq \Cu_1(\CatCasr-\lim\limits_{\longrightarrow}((A_i,\phi_{ij})))\simeq \gamma^\sim(\W^\sim-\lim\limits_{\longrightarrow}(\W_1(A_i),\W_1(\phi_{ij}))).
\]
\end{cor}

\subsection{Algebraic \texorpdfstring{$\Cu^\sim$}{Cu-tilde} semigroups and \texorpdfstring{$\CatoM$}{Mon}-completion}
\label{sec:Part V}
In this last subsection, we will briefly introduce algebraic $\Cu^\sim$-semigroups in order to link the notion of real rank zero for a $\CatCa$-algebra $A$ of stable rank one, that ensures an abundance of projections, with the notion of \textquoteleft density\textquoteright\, of compact elements in $\Cu_1(A)$. In fact, as compact elements of $\Cu_1(A)$ are entirely determined by the ones of its positive cone $\Cu(A)$ (see \autoref{cor:compactcu1}), all results from $\Cu(A)$ will apply here. These can be found in \cite[\S 5.5]{APT14}.\\

Let $S\in\Cu^\sim$. We denote by $S_c:=\{s\in S\mid s\ll s\}$. It is easily shown that $S_c\in\CatoM$ and that for any $\Cu^\sim$-morphism $f:S\longrightarrow T$ between $S,T\in \Cu^\sim$, we have $f(S_c)\subset T_c$. Thus, $f$ induces a $\CatoM$-morphism $f_c:=f_{|S_c}:S_c\longrightarrow T_c$.
Hence, alike $\nu_+$ that recovers the positive cone of a $\Cu^\sim$-semigroup (see \autoref{prop:nu+}), we obtain a functor $\nu_c$ that recovers the compact elements of a $\Cu^\sim$-semigroup:
\[
	\begin{array}{ll}
		\nu_c: \Cu^\sim \longrightarrow \CatoM\\
		\hspace{1cm} S \longmapsto S_c\\
		\hspace{1,05cm} f \longmapsto f_c
	\end{array}
\]

Conversely, let $M\in\CatoM$. Then, $\leq$ is a natural transitive binary relation on $M$ such that $(M,\leq)\in\CatW^\sim$. We denote $\Cu^\sim(M):=\gamma^\sim(M,\leq)$ the $\Cu^\sim$-completion of $(M,\leq)$. Any $\CatoM$-morphism $f:M\longrightarrow N$ between $M,N\in\CatoM$ induces a $\Cu^\sim$-morphism $\gamma^\sim(f): \gamma^\sim(M)\longrightarrow \gamma^\sim(N)$. Thus we obtain a functor:
\[
	\begin{array}{ll}
		\Cu^\sim: \CatoM \longrightarrow \Cu^\sim\\
		\hspace{1,45cm} M \longmapsto \Cu^\sim(M)\\
		\hspace{1,55cm} f \longmapsto \gamma^\sim(f)
	\end{array}
\]

\begin{dfn}
Let $S\in \Cu^\sim$. We say that $S$ is an \emph{algebraic $\Cu^\sim$-semigroup} if every element in $S$ is the supremum of an increasing sequence of compact elements, that is, an increasing sequence in $S_c$. We denote by $\Cu^\sim_{alg}$ the full subcategory of $\Cu^\sim$ consisting of algebraic $\Cu^\sim$-semigroups (see \cite[\S 5.5]{APT14}).
\end{dfn}

\begin{prop} (cf \cite[Proposition 5.5.4]{APT14})

(i) Let $M\in \CatoM$. Then $\Cu^\sim(M)$ is an algebraic $\Cu^\sim$-semigroup and, moreover, there is a natural identification between $M$ and the ordered monoid of compact elements of $\Cu^\sim(M)$.

(ii) For any algebraic $\Cu^\sim$-semigroup $S$, we have $\Cu^\sim(S_c)\simeq S$ as $\Cu^\sim$-semigroups.

\end{prop}

\begin{prop} (cf \cite[Corollary 5]{CEI08}, \cite[Remark 5.5.2]{APT14}).
Whenever $A$ has real rank zero, $\Cu(A)$ is an algebraic $\Cu$-semigroup. If moreover $A$ has stable rank one, then the converse is true.
\end{prop}

\begin{cor}
Let $A$ be a $\CatCa$-algebra of stable rank one. Then $A$ has real rank zero if and only if $\Cu_1(A)\in\Cu^\sim_{alg}$ if and only if $\Cu(A)\in \Cu_{alg}$.
\end{cor}

\begin{proof}
Using the characterization of compacts elements of $\Cu_1(A)$ by compact elements of $\Cu(A)$ as in \autoref{cor:compactcu1}, we get that $\Cu(A)$ is algebraic if and only if $\Cu_1(A)$ is algebraic.
\end{proof}

We end this section by observing that $\nu_+$ and $\nu_c$ satisfy the following: $\nu_+\circ\nu_c\simeq \nu_c\circ\nu_+$. Hence, we sometimes consider $\nu_{+,c}:\Cu^\sim\longrightarrow\CatPoM$ as the composition of $\nu_+$ and $\nu_c$. Naturally, for any $S\in\Cu^\sim$, we denote by $S_{+,c}:=\nu_{+,c}(S)$ the positively ordered monoid of positive compact elements of $S$.

\section{Some computations \texorpdfstring{$\Cu_1$}{Cu1}-semigroups}
\label{sec:Structure}
This section is aiming to compute the unitary Cuntz semigroup of certain $\CatCa$-algebras, such as simple $\CatCa$-algebras of stable rank one, $\AF$ algebras, and some $\A\!\T,\AI$ algebras. We first give another picture of the $\Cu_1$-semigroup and its morphisms using the lattice of ideals of the $\CatCa$-algebra that makes these computations easier.

\subsection{Alternative picture of the invariant}
We start by recalling some well-known facts about (closed two-sided) ideals of a $\CatCa$-algebra. Let $A$ be a $\CatCa$-algebra, the set of closed two-sided ideals, that we write $\Lat(A)$, has a complete lattice structure given by $I\wedge J= I\cap J$ and $I\vee J = I+J$. Furthermore, it has been pointed out \cite[Section 5.1]{APT14} that the set of closed two-sided ideals that contain a full, positive element, that we write $\Lat_f(A)$, is also of an interest since it is not only a sublattice of $\Lat(A)$ but also a $\Cu$-semigroup. In fact, there exists a complete lattice isomorphism between $\Lat(A)$ and $\Lat(\Cu(A))$ that maps $\Lat_f(A)$ onto $\Lat_f(\Cu(A))$, where $\Lat_f(\Cu(A))$ denotes the sublattice of singly-generated $\Cu$-ideals in $\Cu(A)$. (See \cite[Section 5.1]{APT14} and \autoref{prg:latticecu} for more details.) 

It is not hard to see that any $\sigma$-unital ideal belongs to $\Lat_f(A)$, and the converse is not true in general. However, in order to construct the alternative picture of the unitary Cuntz semigroup, we will need the extra-hypothesis that $\Lat_f(A)=\{\sigma\text{-unital ideals of }A\}$. Observe that if $A$ is a separable $\CatCa$-algebra, then $A$ satisfies this extra-hypothesis.

Let $A$ be a $\CatCa$-algebra of stable rank one such that $\Lat_f(A)=\{\sigma\text{-unital ideals of }A\}$ and let $a\in (A\otimes\mathcal{K})_+$. Recall that for any $a\in A_+$, we write $I_a:=\overline{AaA}$ the ideal generated by $a$ and $\her(a):=\overline{aAa}$ the hereditary subalgebra generated by $a$.
 Then $a$ is obviously a full positive element in $I_a$. By the hypothesis can find a strictly positive element of $I_a$, that we write $s_a$. Since $a\in \her (s_a)$, we know that $a\lesssim_{\Cu} s_a$. Observe that the canonical inclusion $i:\her(a)\hooklongrightarrow \her (s_a)=I_a$ is one of our standard morphisms (see \autoref{rmk:canoinj}). That is, in the notation of \autoref{rmk:canoinj}, $\chi_{as_a}=\K_1(i)$. 
Furthermore, using \cite[Theorem 2.8]{Br}, we deduce that $\chi_{as_a}:\K_1(\her(a))\simeq\K_1(I_a)$ is in fact an abelian group isomorphism and $\chi_{as_a}([u]_{\K_1(\her(a))})=[u]_{\K_1(I_a)}$ for any unitary element $u\in\her(a)^\sim$. 

\begin{prop}
\label{prg:delta}
Let $A$ be a $\CatCa$-algebra of stable rank one such that $\Lat_f(A)=\{\sigma\text{-unital ideals of }A\}$. Let $a,b\in (A\otimes\mathcal{K})_+ $ be such that  $a\lesssim_{\Cu}b$. Let $s_a,s_b$ be strictly positive elements of the ideals $I_a,I_b$ respectively. Then the following diagram is commutative: 
\vspace{0cm}\[
\xymatrix{
\mathcal{U}(\her(a)^\sim)\ar[r]^{}\ar[d]_{\theta_{ab}^\sim} & \K_1(\her(a))\ar[d]^{\chi_{ab}}\ar[r]^{\simeq}_{\chi_{as_a}} & \K_1(I_a)\ar[d]^{\chi_{s_as_b}} \\
\mathcal{U}(\her(b)^\sim)\ar[r]_{} & \K_1(\her(b))\ar[r]_{\simeq}^{\chi_{bs_b}} &  \K_1(I_b)
} 
\]
In particular, for any other strictly positive element $s_{a'}$ of $I_a$, we have $\her (s_a)= \her (s_{a'})$ and hence $\chi_{s_as_{a'}}=id_{\K_1(I_a)}$, which finally gives us $\chi_{as_a}=\chi_{as_{a'}}$.
\end{prop}

\begin{proof}
By definition, $\chi_{ab}:=\K_1(\theta_{ab}^\sim)$ and hence the left-square is commutative. Furthermore, by transitivity of $\lesssim_1$ (see \autoref{rmk:canoinj2}), we know that $\chi_{s_as_b}\circ\chi_{as_a}=\chi_{as_b}=\chi_{bs_b}\circ\chi_{ab}$. That is, the right square is commutative, which ends the proof.
\end{proof}

\begin{ntn}
\label{dfn:delta}
Let $A$ be a $\CatCa$-algebra of stable rank one such that $\Lat_f(A)=\{\sigma\text{-unital ideals of }A\}$. Let $a\in (A\otimes\mathcal{K})_+$ and let $s_a$ be any strictly positive element of $I_a$. By \autoref{prg:delta}, $\chi_{as_a}: \K_1(\her(a))\simeq \K_1(I_a)$ is a well-defined group isomorphism that does not depend on the strictly positive element $s_a$ chosen. We write $\delta_a:=\chi_{as_a}$.

Let $I,J\in\Lat_f(A)$ and let $s_I,s_J$ be any strictly positive elements of $I,J$ respectively. Suppose that  $I\subseteq J$ or, equivalently $[s_I]\leq [s_J]$ in $\Cu(A)$. By \autoref{prg:delta}, $\chi_{s_Is_J}: \K_1(I)\longrightarrow \K_1(J)$ is a well-defined group morphism that does not depend on the strictly positive elements chosen. We write $\delta_{IJ}:=\chi_{s_Is_J}$. 
Observe that $\delta_{IJ}=\K_1(i)$, where $i:I\hooklongrightarrow J$ is the canonical inclusion. In particular, $\delta_{II}=id_{\K_1(I)}$. 
\end{ntn}

\begin{prop}
\label{cor:deltas}
Let $A$ be a $\CatCa$-algebra of stable rank one such that $\Lat_f(A)=\{\sigma\text{-unital ideals of }A\}$. Let $a,b\in (A\otimes\mathcal{K})_+ $ such that $[a]\leq [b]$ in $\Cu(A)$. Let $u,v$ be unitary elements of $\her(a)^\sim,\her(b)^\sim$ respectively. We write $[u]:=[u]_{\K_1(\her(a))}$ and $[v]:=[v]_{\K_1(\her(b))}$. Then the following are equivalent: 

(i) $\theta_{ab}^\sim(u)\sim_h v$ in $\her(b) ^\sim$.

(ii) $\chi_{ab}([u])=[v]$ in $\K_1(\her(b))$.

(iii) $\delta_{I_aI_b}(\delta_a([u]))=\delta_b([v])$ in $\K_1(I_b)$, that is, $\delta_{I_aI_b}([u]_{\K_1(I_a)})=[v]_{\K_1(I_b)}$.
\end{prop}

\begin{proof}
Since $\K_1(\theta_{ab}^\sim)=\chi_{ab}$, we trivially obtain that (i) is equivalent to (ii). 
Furthermore, by the right-square of the commutative diagram in \autoref{prg:delta}, we know that $\delta_{I_aI_b}\circ\delta_a([u])=\delta_b\circ\chi_{ab}([u])$. And since $\delta_b$ is an isomorphism, we obtain that (ii) is equivalent to (iii).  
\end{proof}

\vspace{0,5cm}\begin{cor}
\label{prg:deltas}
Let $A$ be a $\CatCa$-algebra of stable rank one and let $[(a,u)],[(b,v)]\in\Cu_1(A)$. Then $[(a,u)]\leq [(b,v)]$ in $\Cu_1(A)$ if and only if
\vspace{-0,6cm}\[
\left\{
    \begin{array}{ll}
    		[a]\leq[b] \text{ in } \Cu(A)\\
        \vspace{0cm}\delta_{I_aI_b}([u]_{\K_1(I_a)})=[v]_{\K_1(I_b)} \text{ in } \K_1(I_b)    \end{array}
\right.
\]
where $\delta_{I_aI_b}$ is as in \autoref{prg:delta}.
\end{cor}

We will now use all the above to get a new picture of the $\Cu_1$-semigroup and its elements. 

\begin{dfn}
\label{dfn:cufull}
Let $A$ be a $\CatCa$-algebra of stable rank one. Let $I\in\Lat_f(A)$ be an ideal of $A$ that contains a full positive element. We recall that $\Cu(I)$ is a singly-generated ideal of $\Cu(A)$. We also recall that for $x\in\Cu(A)$, we write $I_x:=\{y\in\Cu(A) \mid y\leq \infty x\} $ the ideal of $\Cu(A)$ generated by $x$. 

Define $\Cu_f(I):=\{[a] \in\Cu(A) \mid I_a=I\}$. Equivalently, $\Cu_f(I):=\{x\in \Cu(A) \mid I_x=\Cu(I)\}$. In other words, $\Cu_f(I)$ consists of the elements of $\Cu(A)$ that are full in $\Cu(I)$.
\end{dfn}

One could define $\Cu_f(I)$ for any ideal $I\in\Lat(A)$. However, it is easily seen that $\Cu_f(I)\neq\emptyset$ if and only if $I\in\Lat_f(A)$. We also mention that whenever $A$ is separable, we have that $\Lat_f(A)=\Lat(A)$.

For notational purposes, we will indistinguishably use $I_a$ or $I_{[a]}$, referring to one or the other; see \autoref{prg:latticecu}. For instance, we might consider objects such as $\delta_{I_xI_y}$ or $\K_1(I_x)$, where $x,y\in\Cu(A)$, when we really mean $\delta_{I_aI_b}$ or $\K_1(I_a)$, where $a,b\in (A\otimes\mathcal{K})_+$ are representatives of $x,y$ respectively.

\begin{dfn}
\label{dfn:newpicture}
Let $A$ be a $\CatCa$-algebra of stable rank one such that $\Lat_f(A)=\{\sigma\text{-unital ideals of }A\}$. Let us consider 
\vspace{-0,3cm}\[
\hspace{0cm}S:= \bigsqcup\limits_{I\in\Lat_f(A)} \Cu_f(I)\times \K_1(I).
\vspace{0,1cm}\] 
We equip $S$ with addition and order as follows: For any $(x,k)\in \Cu_f(I_x)\times \K_1(I_x)$ and $(y,l)\in\Cu_f(I_y)\times \K_1(I_y)$, then \[
\left\{\begin{array}{ll}
(x,k)\leq (y,l) \text{ if: } x\leq y \text{ and } \delta_{I_xI_{y}}(k)=l.\\
(x,k)+(y,l)=(x+y,\delta_{I_xI_{x+y}}(k)+\delta_{I_yI_{x+y}}(l)).
\end{array}\right.
\]
\end{dfn}

\begin{lma}
\label{lma:PoMisoCu}
Let $S$ be a $\Cu^\sim$-semigroup and let $T$ be a $\CatoM$. Let $f:S\longrightarrow T$ be a $\CatoM$-isomorphism. Then, $T$ is a $\Cu^\sim$-semigroup and $f$ is a $\Cu^\sim$-isomorphism. A fortiori, $S\simeq T$ as $\Cu^\sim$-semigroups.
\end{lma}

\begin{proof}
We recall that suprema and the compact-containment relation are entirely determined by the order-structure. Thus, existence of suprema and axioms (O1)-(O4) in $T$ are directly obtained from the surjective order-embedding $f$. More concretely, for any increasing sequence $(t_k)_k$ in $T$, we have that $\sup\limits_k t_k=f(\sup\limits_k s_k)$ where $s_k$ is the (unique) element in $S$ such that $f(s_k)=t_k$. It is now routine to check that $T$ is a $\Cu^\sim$-semigroup and that $f$ is a $\Cu^\sim$-isomorphism.
\end{proof}

\begin{thm}
\label{prop:newpicture}
Let $A$ be a $\CatCa$-algebra of stable rank one such that $\Lat_f(A)=\{\sigma\text{-unital ideals of }A\}$. \\Let $(S,+,\leq)$ be the object defined in \autoref{dfn:newpicture}. Then $(S,+,\leq)$ is a $\Cu^\sim$-semigroup and the following map is a $\Cu^\sim$-isomorphism:
\vspace{-0,1cm}\[
	\begin{array}{ll}
		\xi:\Cu_1(A) \longrightarrow S\\
		\hspace{0,5cm}[(a,u)] \longmapsto ([a],\delta_a([u]))
	\end{array}
\] where $[a]:=[a]_{\Cu(A)}$ and $[u]:=[u]_{\K_1(\her(a))}$. 
\end{thm}

\begin{proof}
By \autoref{dfn:delta} and \autoref{dfn:cufull}, the map $\Cu_1(A) \longrightarrow \bigsqcup\limits_{I\in\Lat_f(A)} \Cu_f(I)\times \K_1(I)
$ is well-defined. Further, by construction, addition and order are well-defined in $S$. Now let $a\in (A\otimes\mathcal{K})_+$. Since $A$ has stable rank one, then so has $\her(a)$. Hence, by $\K_1$-surjectivity, we know that any element of $\K_1(\her(a))$ lifts to a unitary in $\her(a)^\sim$ and that any two of those lifts are homotopic. Also $\delta_a$ is an isomorphism and obviously any two representatives of $x$ in $(A\otimes\mathcal{K})_+$ are Cuntz equivalent. Thus for any $(x,k)\in\Cu(A)\times \K_1(I_x)$, there exist $a\in (A\otimes\mathcal{K})_+ $ and $u\in\mathcal{U}(\her(a)^\sim)$ such that $[a]=x$ and $\delta_a[u]=k$. Moreover, for any other lift $(a',u')$, we have $[(a',u')]=[(a,u)]$. So we conclude that $\xi$ is a set bijection.

Now, using \autoref{cor:deltas} and \autoref{prg:deltas}, we know that $[(a,u)]\leq[(b,v)]$ if and only if $\xi([(a,u)])\leq \xi([(b,v)])$. Moreover, using \autoref{prg:delta}, we have $\xi([(a,u)]+[(b,v)])=\xi([(a,u)])+\xi([(b,v)])$. In the end, we obtain that $\xi$ is a $\CatoM$-isomorphism. We finally conclude that $S$ is a $\Cu^\sim$-semigroup and that $\xi$ is a $\Cu^\sim$-isomorphism using \autoref{lma:PoMisoCu}.
\end{proof}

In this new picture, the positive elements of $\Cu_1(A)$ can be identified with $\{(x,0)\mid x\in\Cu(A)\}$ (see \autoref{prop:nu+}). In other words, $\Cu_1(A)_+\simeq \Cu(A)$ as $\Cu$-semigroups. We will end this part by describing morphisms from $\Cu_1(A)$ to $\Cu_1(B)$ in this new viewpoint of our invariant.

\begin{lma}
\label{lma:newpicturemorph}
Let $A,B$ be $\CatCa$-algebras of stable rank one such that $\Lat_f(A)=\{\sigma\text{-unital ideals of }A\}$. \\Let $I\in\Lat_f(A)$ and let $\phi: A\longrightarrow B$ be a $^*$-homomorphism. Write $J:=\overline{B\phi(I)B}$ the smallest ideal of $B$ containing $\phi(I)$. Also write $\alpha:=\Cu_1(\phi),\alpha_0:=\Cu(\phi)$ and $\alpha_I:=\K_1(\phi_{|I})$, where $\phi_{|I}:I\overset{\phi}\longrightarrow J$.

(i) For any $x\in\Cu_f(I)$, we have $\alpha_0(x)\in\Cu_f(J)$. That is, $I_{\alpha_0(x)}=\Cu(J)$ is the smallest ideal of $\Cu(B)$ containing $\alpha_0(\Cu(I))$ and $\Cu(J)\in\Lat_f(\Cu(B))$. 

(ii) For any $(x,k)$ with $x\in\Cu_f(I)$ and $k\in \K_1(I)$, we have $\alpha(\xi^{-1}(x,k))=(\alpha_0(x),\alpha_I(k))$, where $\xi_A, \xi_B$ are the $\Cu^\sim$-isomorphism as in \autoref{prop:newpicture} for $A,B$ respectively. 
\end{lma}

\begin{proof}
By functoriality of $\Cu$ and \autoref{prg:latticecu}, we know that $\Cu(J)$ is the smallest ideal of $\Cu(B)$ that contains $\alpha_0(\Cu(I))$. Now let $x\in\Cu_f(I)$. Then $\alpha_0(x)\in\alpha_0(\Cu(I))$. Hence $I_{\alpha_0(x)}\subseteq\Cu(J)$. However, since  $x$ is full in $\Cu(I)$, we have $\alpha_0(\Cu(I))\subseteq I_{\alpha_0(x)}$. By minimality of $\Cu(J)$ we deduce that $I_{\alpha_0(x)}=\Cu(J)$, that is, $\alpha_0(x)\in \Cu_f(J)$, which proves (i).

(ii) Let $(x,k)$ be an element of $\Cu_1(A)$, where $x\in\Cu(A)$ and $k\in \K_1(I_x)$. Let $(a,u)$ be a representative of $(x,k)$, that is, $\xi([(a,u)])=(x,k)$. We know that
\begin{align*}
\alpha(\xi^{-1}(x,k))&=\alpha([a,u])\\
			  &=[(\phi(a)),\phi^\sim(u))]\\		
			  &=([\phi(a)]_{\Cu(B)},\delta_{\phi(a)}([\phi^\sim(u)]_{\K_1(\her(\phi(a))^\sim)}))\\
			  &= ([\phi(a)]_{\Cu(B)},[\phi^\sim(u)]_{\K_1(I_{\phi(a)}^\sim)})
\end{align*}
Hence $\alpha(\xi^{-1}(x,k))=(\alpha_0(x),\alpha_I(k))$ as desired.
\end{proof}

\begin{ntn}\label{prg:newpicture}
Whenever convenient, and many times in the sequel, we will describe elements of $\Cu_1(A)$ as a pair $(x,k)$ where $x\in \Cu(A)$ and $k\in \K_1(I_x)$, whenever $A$ is a separable $\CatCa$-algebra of stable rank one. Again, we may describe morphisms $\alpha:=\Cu_1(\phi)$ from $\Cu_1(A)$ to $\Cu_1(B)$, as pairs $\alpha:=(\alpha_0,\{\alpha_I\}_{I\in\Lat_f(A)})$, where $\alpha_0:=\Cu(\phi)$ and $\alpha_I:=\K_1(\phi_{|I})$.
\end{ntn}

We now compute the $\Cu_1$-semigroup in some specific settings. In the process, we will remind the reader about lower semicontinuous functions which play a key role in the computation of $\Cu$-semigroups of certain $\CatCa$-algebras. 

\begin{prg} \textbf{(Lower semicontinuous functions.)}
\label{prg:lsc}
Let $X$ be a topological space and $S$ be a $\Cu$-semigroup. Let $f:X\longrightarrow S$ be a map. We say that $f$ is \emph{lower semicontinuous} if for any $s\in S$, the set $\{t\in X \mid s\ll f(t)\}$ is open in $X$. We write $\Lsc(X,S)$ for the set of lower-semicontinuous functions from $X$ to $S$. 

Also, we recall that if $A$ is a separable $\CatCa$-algebra of stable rank one such that $\K_1(I)=0$ for every ideal of $A$ and $X$ is a locally compact Hausdorff space that is second countable and of covering dimension at most one, then $\Cu(\mathcal{C}_0(X)\otimes A)\simeq \Lsc(X,\Cu(A))$; see \cite[Theorem 3.4]{AntoinePereraSantiago11}.

Finally, $V\longmapsto I_{1_{V}}$ defines a one-to-one correspondence between the open subsets of $X$, that we write $\mathcal{O}(X)$, and the ideals of $\Lsc(X,\overline{\N})$. Note that for any $f\in\Lsc(X,\overline{\N})$, $I_f:=I_{\supp (f)}$, where $\supp (f):=\{x\in X \mid f(x)\neq 0\}$ is an open set of $X$.
\end{prg}
\vspace{-0,2cm}
\subsection{The simple case}

\label{prop:computesimple}
Let $A$ be a simple $\sigma$-unital $\CatCa$-algebra of stable rank one. Then $\Cu_1(A)$ can be described in terms of $\Cu(A)$ and $\K_1(A)$ as follows: 
\vspace{-0,1cm}\[
	\begin{array}{ll}
		\Cu_1(A)\overset{\simeq}{\longrightarrow} ((\Cu(A)\setminus\{0\})\times \K_1(A))\sqcup \{0\}\\
		\hspace{0,3cm} (x,k)\longmapsto \left\{\begin{array}{ll} 0 \text{ if } x=0\\
				(x,k) \text{ otherwise }
				\end{array}
			\right.
	\end{array}
\]
\begin{proof}
Since $A$ is simple, we know that $\Lat(A)=\{0,A\}$. Therefore, in the description of the $\Cu_1$-semigroup of \autoref{prg:newpicture}, we have $\Cu_f(\{0\})=\{0\}$ and $\Cu_f(A)=\Cu(A)\setminus\{0\}$. The result follows. 
\end{proof}

\subsection{The case of no \texorpdfstring{$\K_1$}{K1}-obstructions}

\begin{dfn}
We say that a $\CatCa$-algebra $A$ has \emph{no $\K_1$-obstructions}, if $A$ has stable rank one and $\K_1(I)$ is trivial for any $I\in\Lat(A)$.
\end{dfn}

\begin{prop}
Let $A$ be a $\CatCa$-algebra with no $\K_1$-obstructions such that $\Lat_f(A)=\{\sigma\text{-unital ideals of }A\}$. Then $\Cu_1(A)\simeq \Cu(A)$. In particular, for any separable $\AF$ algebra $A$, $\Cu_1(A)\simeq \Cu(A)$.
\end{prop}

\begin{proof}
By assumption, we know that $\K_1(I)$ is trivial for any $I\in\Lat(A)$. Therefore, using again the description of the $\Cu_1$-semigroup of \autoref{prop:newpicture} (see \autoref{prg:newpicture}), we have $\Cu_1(A)\simeq \Cu(A)\times\{0\}$. The result follows. 
\end{proof}

\subsection{\texorpdfstring{$\AI$}{AI} and \texorpdfstring{$\A\!\T$}{AT} algebras: The case of \texorpdfstring{$\mathcal{C}([0,1])$}{C([0,1])} and \texorpdfstring{$\mathcal{C}(\T)$}{C(T)}}
Here we compute the $\Cu_1$-semigroup of the interval algebra and the circle algebra. Using the continuity of $\Cu_1$, we also give an explicit computation of the $\Cu_1$-semigroup of $\AI$-algebras (respectively $\A\!\T$-algebras), constructed as the tensor product of the interval algebra (respectively the circle algebra) with any UHF algebra of infinite type. 

\begin{ntn}
\label{prg:disjointunioninterval}
Let $X$ be the interval or the circle and let $f\in\Lsc(X,\overline{\N})$. The open set $V_f:=\supp f$ of $X$ can be (uniquely) decomposed into a countable disjoint union $\{V_k\}_k$ of open arcs of $X$. In other words, $V_f = \bigsqcup\limits_{k=1}^{n_f}V_k$, for some $n_f\in\overline{\N}$, where $\{V_k\}_{k\geq 1}$ are pairwise disjoint open arcs of $X$. For the specific case of the interval, we also define $m_f:=n_{f}-(1_{V_f}(0)+ 1_{V_f}(1))$. That is, $m_f$ is the number of open intervals of the decomposition $\{V_k\}_{k\geq 1}$ of $V_f$ that are strictly contained in $]0,1[$. (Therefore $m_f$ also belongs to $\overline{\N}$.)
\end{ntn}

\textbf{The $\mathcal{C}([0,1])$ case.}

\begin{lma}
\label{lma:computeinterval}
Let $I\in\Lat(\mathcal{C}([0,1]))$ and let $f_I:=1_{V_I}$ be the indicator map on the unique open set $V_I$ of $[0,1]$ corresponding to $I$. We have \vspace{-0,6cm}\[
\left\{\begin{array}{ll}
\Cu_f(I)\simeq \Lsc(V_I,\overline{\N}_*).\\
 \K_1(I)\simeq\underset{m_{f_I}}{\oplus}\Z.
\end{array}\right.
\]
\end{lma}

\begin{proof}
We know that $\Cu(I)=I_{f_I}\simeq \Lsc(V_I,\overline{\N})$ and we obtain that $\Cu_f(I)\simeq \Lsc(V_I,\overline{\N}_*)$.  Then, we observe that open arcs of $[0,1]$ are of the following the form:
\vspace{0cm}\[
\,]a,b[\hspace{0,5cm} [0,1] \hspace{0,5cm} ]a,1] \hspace{0,5cm} [0,a[\hspace{0,5cm}  \emptyset
\]
and the $\K_1$ groups of continuous maps over these open arcs are, respectively:
\vspace{0cm}\[
\hspace{0,55cm}\Z \hspace{0,85cm} \{0\} \hspace{0,85cm} \{0\} \hspace{0,8cm}\{0\} \hspace{0,7cm} \{0\}
\] 
Furthermore, for any two disjoint open arcs $V,W$ of the open interval $]0,1[$, the canonical inclusion $i:I_{1_V}\subseteq I_{1_V+1_W}$ induces an injection $K_1(i):\Z\overset{id\oplus 0}{\lhook\joinrel\longrightarrow} \Z\oplus\Z$. Now, let us decompose $V_I= \bigsqcup\limits_{k=1}^{n_{f_I}}V_k $ as in \autoref{prg:disjointunioninterval}. Equivalently, we have that $f_I=\sum\limits_{k=1}^{n_{f_I}} 1_{V_k}$. Using all the above, we compute that $\K_1(I)\simeq \lim\limits_{n\geq 1}(\K_1(I_{(\sum\limits_{k=1}^{n} 1_{V_k})}),\K_1(i_n))\simeq \underset{m_{f_I}}{\oplus}\Z$, where $i_n:I_{(\sum\limits_{k=1}^{n} 1_{V_k})}\subseteq I_{(\sum\limits_{k=1}^{n+1} 1_{V_k})}$ is the canonical inclusion.
\end{proof}

\begin{thm}
Let $W_0:=[0,1[$ and $W_1:=]0,1]$. Then:

(i)\vspace{-0,7cm}\[
	\hspace{2cm}\begin{array}{ll}
		\Cu_1(\mathcal{C}([0,1]))\simeq \bigsqcup\limits_{V\in \mathcal{O}([0,1]))} \Lsc(V,\overline{\N}_*)\times (\underset{m_{1_V}}{\oplus}\Z)\\
		\hspace{2,08cm}\simeq \Cu_1(\mathcal{C}(]0,1[))\sqcup(\bigsqcup\limits_{i=0,1}\Lsc(W_i,\overline{\N}_*)\times \{0\}) \sqcup \Lsc([0,1],\overline{\N}_*)\times\{0\}.
	\end{array}
\]

\vspace{-0,2cm}(ii) $\Cu_1(\mathcal{C}([0,1]))_c\simeq(\{n 1_{[0,1]}\}_{n\in \N})\times\{0\}\simeq \N$.
\end{thm}

\begin{proof}
(i) Combine \autoref{prop:newpicture} with \autoref{lma:computeinterval} and \autoref{prg:lsc}. 

(ii) From \autoref{cor:compactcu1}, we know that $(x,k)\in\Cu_1(\mathcal{C}([0,1]))$ is a compact element if and only if $x$ is compact in $\Lsc([0,1],\overline{\N})$, if and only if $x$ is constant on $[0,1]$ and $x\ll \infty$.
\end{proof}

\textbf{The $\mathcal{C}(\T)$ case.}

\begin{lma}
\label{lma:computecircle}
Let $I\in\Lat(\mathcal{C}(\T))$ and let $f_I:=1_{V_I}$ be the indicator map on the unique open set $V_I$ of $\T$ corresponding to $I$. We have \vspace{-0,3cm}\[
\left\{\begin{array}{ll}
\Cu_f(I)\simeq \Lsc(V_I,\overline{\N}_*).\\
 \K_1(I)\simeq\underset{n_{f_I}}{\oplus}\Z.
\end{array}\right.
\]
\end{lma}

\begin{proof}
We know that $\Cu(I)=I_{f_I}\simeq \Lsc(V_I,\overline{\N})$ and we obtain that $\Cu_f(I)\simeq \Lsc(V_I,\overline{\N}_*)$.  Then, we observe that open arcs of $\T$ are of the following the form:
\vspace{-0cm}\[
\,]a,b[\hspace{0,5cm} \T \hspace{0,5cm}  \emptyset
\]
and the $\K_1$ groups of continuous maps over these open arcs are, respectively:
\vspace{-0cm}\[
\hspace{0,5cm}\Z \hspace{0,8cm} \Z \hspace{0,5cm} \{0\}
\] 
Furthermore, for any two disjoint open arcs $V,W$ of $\T$, the canonical inclusion $i:I_{1_V}\subseteq I_{1_V+1_W}$ induces a injection $K_1(i):\Z\overset{id\oplus 0}{\lhook\joinrel\longrightarrow} \Z\oplus\Z$. Now, let us decompose $V_I= \bigsqcup\limits_{k=1}^{n_{f_I}}V_k $ as in \autoref{prg:disjointunioninterval}. Equivalently, we have that $f_I=\sum\limits_{k=1}^{n_{f_I}} 1_{V_k}$. Using all the above, we get $\K_1(I)\simeq \lim\limits_{n\geq 1}(\K_1(I_{(\sum\limits_{k=1}^{n} 1_{V_k})}),\K_1(i_n))\simeq \underset{n_{f_I}}{\oplus}\Z$, where $i_n:I_{(\sum\limits_{k=1}^{n} 1_{V_k})}\subseteq I_{(\sum\limits_{k=1}^{n+1} 1_{V_k})}$ is the canonical inclusion.
\end{proof}

\begin{thm}
\label{thm:computecircle}
We have the following:

(i)\vspace{-0,45cm}\[
	\hspace{0cm}\begin{array}{ll}
		\Cu_1(\mathcal{C}(\T))\simeq \bigsqcup\limits_{V\in \mathcal{O}(\T)} \Lsc(V,\overline{\N}_*)\times (\underset{n_{1_V}}{\oplus}\Z)\\
		\hspace{1,56cm}\simeq \Cu_1(\mathcal{C}(]0,1[)) \sqcup \Lsc(\T,\overline{\N}_*)\times\Z.
	\end{array}
\]

(ii) $\Cu_1(\mathcal{C}(\T))_c\simeq(\{n 1_{ \T}\}_{n\in\N})\times\Z\simeq \N\times\Z$. 
\end{thm}

\begin{proof}
(i) Combine \autoref{prop:newpicture} with \autoref{lma:computeinterval} and \autoref{prg:lsc}. 

(ii) From \autoref{cor:compactcu1}, we know that $(x,k)\in\Cu_1(\mathcal{C}(\T))$ is a compact element if and only if $x$ is compact in $\Lsc(\T,\overline{\N})$, if and only if $x$ is constant on $\T$ and $x\ll \infty$.
\end{proof}

Now that we have computed the $\Cu_1$-semigroup of the interval algebra and the circle algebra, we are able to obtain the $\Cu_1$-semigroup of any $\AI$ and $\A\!\T$ algebra, using \autoref{cor:cu1ctinuous}. Actually, we will next compute a concrete example of an $\A\!\T$ algebra that is constructed as $\mathcal{C}(\T)\otimes$UHF. 

Let $q$ be a supernatural number and consider $M_q$ the UHF algebra associated to $q$. Consider any sequence of prime numbers $(q_n)_n$ such that $q= \prod\limits_{n\in \N}q_n $. Write $(A_n,\phi_{nm})_n$ the inductive system associated to $(q_n)_n$. Now consider the following $\A\!\T$ algebra: $A:= \lim\limits_{\longrightarrow_n}(\mathcal{C}(\T)\otimes A_n,id\otimes\phi_{nm})$. In fact, $A\simeq \mathcal{C}(\T)\otimes M_q$. (Similar construction and computations can be done for the interval).

\begin{thm}
Let $M_q$ be a UHF algebra. Then:
\vspace{-0,1cm}\[\Cu_1(\mathcal{C}(\T)\otimes M_q)\simeq \bigsqcup\limits_{V\in \mathcal{O}(\T)} \Lsc(V,(\Cu(M_q)\setminus\{0\}))\times (\underset{n_{1_V}}{\oplus}\K_0(M_q)).\]
In particular, for any UHF algebra of infinite type $M_{p^\infty}$, we get:
\vspace{-0,1cm}\[\Cu_1(\mathcal{C}(\T)\otimes M_q)\simeq \bigsqcup\limits_{V\in\mathcal{O}(\T)}\Lsc(V,(\N[\frac{1}{p}]\sqcup  ]0,\infty]\setminus\{0\}))\times (\underset{n_{1_V}}{\oplus} \Z[\frac{1}{p}]).\]
\end{thm}

\begin{proof}
Since UHF algebras are simple, we know that all ideals of $\mathcal{C}(\T)\otimes M_q$ are of the form $\mathcal{C}_0(U)\otimes M_q$ for some $U\in\mathcal{O}(\T)$. Hence, using the K\"unneth formula (see \cite[Theorem 23.1.3]{Bl86}), we obtain that $\K_1(\mathcal{C}_0(U)\otimes M_q)\simeq (\underset{1}{\overset{n_U}{\oplus}}\Z)\otimes\K_0(M_q)\simeq \underset{1}{\overset{n_U}{\oplus}}\K_0(M_q)$. On the other hand, by \cite[Theorem 3.4]{AntoinePereraSantiago11}, we compute that $\Cu(\mathcal{C}_0(U)\otimes M_q)\simeq \Lsc(U,\Cu(M_q))$. The result follows from \autoref{prop:newpicture}.
\end{proof}

\section{Relation of \texorpdfstring{$\Cu_1$}{Cu1} with existing K-Theoretical invariants}
\label{section:recoverfunctors}
The aim of this section is to recover existing invariants functorially. We have already seen that the positive cone of $\Cu_1(A)$ is isomorphic to $\Cu(A)$. Our first step is to capture the $\K_1$ group information. To that end, we define a well-behaved set of maximal elements of a $\Cu^\sim$-semigroup $S$, written $S_{max}$, and we prove that $\Cu_1(A)_{max}$ is isomorphic to $\K_1(A)$. Subsequently, we recover functorially $\Cu$, $\K_1$ and finally the $\K_*$ group. 
As before, we shall assume that $A$ is a $\CatCa$-algebra that has stable rank one and denote the category of such $\CatCa$-algebras by $\CatCasr$.

\subsection{An abelian group of maximal elements: \texorpdfstring{$\nu_{max}$}{nu max}}

\begin{dfn}
Let $S$ be a $\Cu^\sim$-semigroup. We say that $S$ is \emph{positively directed} if, for any $x\in S$, there exists $p_x\in S$ such that $x+p_x\geq 0$.
\end{dfn}

\begin{lma}
Let $A$ be a $\CatCa$-algebra of stable rank one. Then $\Cu_1(A)$ is positively directed.
\end{lma}

\begin{proof}
Using the picture of \autoref{prop:newpicture} (see \autoref{prg:newpicture}), consider $(x,k)\in\Cu_1(A)$, where $x\in\Cu(A)$ and $k\in \K_1(I_x)$. Observing that $(x,k)+(x,-k)=(2x,0)\geq 0$, we deduce that $\Cu_1(A)$ is positively directed. 
\end{proof}

\begin{dfn}
Let $S$ be a $\Cu^\sim$-semigroup. We define
$S_{max}:=\{x\in S \mid \text{ if } y\geq x, \text{ then } y=x\}$. 
\end{dfn}

\begin{prop}
\label{prop:PCABGP}
Let $S$ be a positively directed $\Cu^\sim$-semigroup. Then $S_{max}$ is either empty or an absorbing abelian group in $S$ whose neutral element $e_{S_{\max}}$ is positive. 
\end{prop}

\begin{proof}
The empty case is trivial. Let us suppose that $S_{max}$ is not empty. By assumption, for any $x\in S$, there exists at least one element $p_x\in S$, such that $x+p_x\geq 0$. We first show that $S_{max}$ is closed under addition.

Let $y,z$ be elements in $S_{max}$ and let $x\in S$ be such that $x\geq y+z$. We first have $x+p_z\geq y+z+p_z\geq y$ and $x+p_y\geq z+y+p_y\geq z$, which gives us the following equalities: $x+p_z=y+z+p_z = y$ and $x+p_y=z+y+p_y= z$. Obviously $x\leq x+p_z+z= x+p_z+x+p_y= y+z$ and since $x\geq y+z$, we have $x=y+z$ which tells us that $S_{max}$ is closed under addition.

Now, let us show that $S_{max}$ has neutral element. We first prove that for any $z\in S_{max}$ and any $p_z\in S$ such that $z+p_z\geq 0$,  the element $z+p_z$ is a positive element of $S_{max}$ that does not depend on $z$ nor $p_z$:

Let $z$ and $p_z$ be such elements and let $x\in S$ be such that $x\geq z+p_z$. We know that for any $y\in S_{max}, \, y+z+p_z=y$. In particular, $2z+p_z=z$. Also, $x+z\geq 2z+p_z=z$. Hence $x+z=z$. Finally compute that $x\leq x+z+p_z=z+p_z$. Therefore $x=z+p_z$, that is, $z+p_z\in S_{max}$. Further, for any $y,z$ elements of $S_{max}$, we have $y+p_y+z+p_z\geq z+p_z$, $y+p_y$, which by what we have just proved gives us $y+p_y= y+p_y+z+p_z=z+p_z$. Hence, the positive element  $e_{S_{max}}:=y+p_y$ belongs to $S_{max}$ and does not depend on $y$ and $p_y$. Now let $z\in S_{max}$. Since $e_{S_{max}}\geq 0$, we obtain $z+e_{max}\geq z$ and we get that $z+e_{S_{max}}=z$ for any $z\in S_{max}$. Thus, $e_{S_{max}}$ is the neutral element for $(S_{max},+)$ and the unique positive element of $S_{max}$. In other words, $S_{max}$ is an abelian monoid with neutral element its unique positive element $e_{S_{max}}$. 

Then on, let us prove that any element has an additive inverse. We already know that $z+(2p_z+z)=e_{S_{max}}$ for any $z\in S_{max}$. Let us show that $2p_z+z$ belongs to $S_{max}$ for any $z\in S_{max}$ and any $p_z\in S$ such that $z+p_z\geq 0$. Let $x\geq 2p_z+z$. Then $x+z\geq e_{S_{max}}$, hence $x+z=e_{S_{max}}$. On the other hand, $x\leq x+z+p_z= e_{S_{max}}+p_z= 2p_z+z$. Therefore $2p_z+z$ belongs to $S_{max}$ and is the (unique) inverse of $z$, which ends the proof that $S_{max}$ is an abelian group. 

Lastly, let us show the absorption property. Let $x\in S$ and let $p\in S_{max}$, we know there exists $y\in S$ such that $x+y\geq 0$. Hence $x+y+p\geq p$. Let $z\in S$ be such that $z\geq x+p$. We have $z+y\geq x+y+p=p$ and hence $z+y=p$. Now since $x+y\geq 0$, we have $z\geq x+p=x+z+y\geq z$ which gives us $z=x+p$, that is, $x+p\in S_{max}$ for any $x\in S$ and $p\in S_{max}$.
\end{proof}

We note that a positively directed $\Cu^\sim$-semigroup $S$ might not have maximal elements. However if it does, then $S$ a unique positive maximal element which is the neutral element for $S_{max}$. Also, whenever $S$ is simple or countably-based, the existence of such a maximal positive element is ensured and $S_{max}$ is not empty. 

As a result, whenever $A$ is either a simple or separable $\CatCa$-algebra of stable rank one, then $\Cu_1(A)_{max}$ is an abelian group whose neutral element is $e_{S_{max}}:=(\infty_{\Cu(A)},0_{\K_1(A)})$. In fact, we will see that under such hypothesis, we have $\Cu_1(A)_{max}\simeq \K_1(A)$.

\begin{prop}
\label{dfn:Sinfinite}
Let $\alpha:S\longrightarrow T$ be a $\Cu^\sim$-morphism between positively directed $\Cu^\sim$-semigroups $S,T$ that have maximal elements.
Then $\alpha_{max}:= \alpha_{|S_{max}} + e_{T_{max}}$ is a $\AbGp$-morphism from $S_{max}$ to $T_{max}$.
\end{prop}

\begin{proof}
Let us first show that $\alpha_{max}$ is a  group morphism. For any $s\in S_{max}$, we know that $(\alpha(s)+e_{T_{max}})\in T_{max}$. Now, since $\alpha$ is a $\Cu^\sim$-morphism, we have $\alpha_{max}(s_1)+\alpha_{max}(s_2)=\alpha(s_1)+\alpha(s_2)+2 e_{T_{max}}=\alpha(s_1+s_2)+e_{T_{max}}=\alpha_{max}(s_1+s_2)$, for any $s_1,s_2$ elements of $S_{max}$. 
\end{proof}

As with $\nu_+$ and $\nu_c$, we define a functor $\nu_{max}$ that recovers the maximal elements of a positively directed $\Cu^\sim$-semigroup as follows:
\vspace{-0,25cm}\[
	\begin{array}{ll}
		\nu_{max}: \Cu^\sim \longrightarrow \AbGp\\
		\hspace{1,26cm} S \longmapsto S_{max}\\
		\hspace{1,3cm} \alpha \longmapsto \alpha_{max}
	\end{array}
\] 

It is left to the reader to prove that $\nu_{max}$ is a well-defined functor. Also, we specify that to be thoroughly defined as a functor, $\nu_{max}$ should have as domain the full subcategory of positively directed $\Cu^\sim$-semigroups that have maximal elements, that we also denote $\Cu^\sim$. Observe that $\Cu_1(\CatCa_{\sr1,\sigma})$ belongs to the latter full subcategory, where $\CatCa_{\sr1,\sigma}$ is the full subcategory of separable $\CatCa$-algebras of stable rank one.

\subsection{Link with \texorpdfstring{$\Cu$}{Cu} and \texorpdfstring{$\K_1$}{K1}}
Recall that for a positively directed $\Cu^\sim$-semigroup $S$ that has maximal elements, we have $S_+\in \Cu$ and that $S_{max}\in\AbGp$; see \autoref{prop:PCABGP}. In fact, both categories $\Cu$ and $\AbGp$ can be seen as subcategories of $\Cu^\sim$ -by defining an order as the equality for the case of groups-.
Therefore, in what follows, we consider $\nu_+$ and $\nu_{max}$ as functors with codomain $\Cu^\sim$.

\begin{dfn}
\label{dfn:vch}
Let $S$ be a positively directed $\Cu^\sim$-semigroup that has maximal elements. Let us define two $\Cu^\sim$-morphisms that link $S$ to $S_+$ on the one hand, and to $S_{max}$ on the other hand, as follows:
\vspace{0cm}\[
	\begin{array}{ll}
		i: S_+\overset{\subseteq}\hooklongrightarrow S \hspace{3cm} j:S \twoheadrightarrow S_{max}\\
		\hspace{0,7cm}s\longmapsto s \hspace{3,5cm}s\longmapsto s+e_{S_{max}}
	\end{array}
\vspace{-0,1cm}\]
\end{dfn}

In the next theorem, we use the picture of the $\Cu_1$-semigroup obtained from \autoref{prop:newpicture} (see \autoref{prg:newpicture}).

\begin{thm}
\label{thm:naturaltransfolma}
Let $A$ be either a separable or a simple $\sigma$-unital $\CatCa$-algebra of stable rank one. We have the following natural isomorphisms in $\Cu$ and $\AbGp$ respectively:
\vspace{-0cm}\[
	\hspace{0cm}\begin{array}{ll}
		\Cu_1(A)_+\simeq \Cu(A) \hspace{2,4cm} \Cu_1(A)_{max}\simeq \K_1(A)\\
		\hspace{0,35cm}(x,0)\longmapsto x \hspace{3,3cm} (\infty_A,k)\longmapsto k
	\end{array}
\]
In fact, we have the following natural isomorphisms: $\nu_+\circ\Cu_1\simeq \Cu$ and $\nu_{max}\circ\Cu_1\simeq \K_1$.
\end{thm}

\begin{proof}
Let us prove the theorem for $A$ separable and the simple case is proven similarly. We know that any positive element of $\Cu_1(A)$ is of the form $(x,0)$ for some $x\in\Cu(A)$ and that $\infty_A:=[s_{A\otimes\mathcal{K}}]=\sup\limits_{n\in\N}n [s_A]$ is the largest element of $\Cu(A)$, where $s_A$ is a strictly positive element of $A$. We also know that any maximal element of $\Cu_1(A)$ is of the form $(\infty_A,k)$ for some $k\in \K_1(A)$. Hence we easily get the two canonical isomorphisms of the statement. 
Now let $\phi:A\longrightarrow B$ be a $^*$-homomorphism, let $(x,0)\in\Cu_1(A)_+$ and let $(\infty_A,k)\in\Cu_1(A)_{max}$. We have that $\Cu_1(\phi)_+(x,0)=(\Cu(\phi)(x),0)$ and that 
\vspace{0cm}\begin{align*}\Cu_1(\phi)_{max}(\infty_A,k)&=(\Cu(\phi)(\infty_A),\Cu_1(\phi)_A(k))+(\infty_B,0)\\
&=(\infty_B,\delta_{I_{\phi(\infty_A)}B}\circ\Cu_1(\phi)_A(k))\\
&=(\infty_B,\K_1(\phi)(k)).
\end{align*} 
This exactly gives us that \[
\xymatrix{
\Cu_1(A)_{+} \ar[d]_{\Cu_1(\phi)_{+}}\ar[r]^{\simeq} & \Cu(A)\ar[d]^{\Cu(\phi)}
&&\Cu_1(A)_{max} \ar[d]_{\Cu_1(\phi)_{max}}\ar[r]^{\simeq} & \K_1(A)\ar[d]^{\K_1(\phi)}
\\
\Cu_1(B)_+\ar[r]_{\simeq} & \Cu(B)&&\Cu_1(B)_{max}\ar[r]_{\simeq} & \K_1(B)
}
\]
are commutative squares.
\end{proof}

\subsection{Recovering an invariant}
\label{sec:ClassificationMachinery}
We will now define the categorical notion of \textquoteleft recovering\textquoteright\, a functor. This allows us to check whether information and classification results of an invariant can be recovered from another one. 
To that end, we introduce the notion of \emph{weakly-complete} invariant: an isomorphism at the level of the codomain category implies an isomorphism at the level of $\CatCa$-algebras without knowing whether it actually corresponds to a lift.

\begin{dfn}

Let $\mathcal{C}, \mathcal{D} $ be arbitrary categories and let $I:\CatCa\longrightarrow \mathcal{C}$ and $J:\CatCa\longrightarrow \mathcal{D}$ be (covariant) functors. Let $H:\mathcal{D}\longrightarrow \mathcal{C}$ be a functor such that there exists a natural isomorphism $\eta:H\circ J\simeq I$. Then we say we can \emph{recover $I$ from $J$ through $H$}.
\end{dfn}

\begin{thm}
\label{thm:recoverfunctor}

Let $\mathcal{C}, \mathcal{D} $ be arbitrary categories and let $I:\CatCa\longrightarrow \mathcal{C}$ and $J:\CatCa\longrightarrow \mathcal{D}$ be (covariant) functors. Suppose that there exists a functor $H:\mathcal{D}\longrightarrow \mathcal{C}$ such that we recover $I$ from $J$ through $H$.

(i) If $I$ is a complete invariant for a class $\CatCa_I$ of $\CatCa$-algebras, then $J$ is a weakly-complete invariant for $\CatCa_I$.

(ii) If $I$ classifies homomorphisms from a class $\CatCa_1$ of $\CatCa$-algebra to another class $\CatCa_2$ of $\CatCa$-algebra, then $J$ weakly classifies homomorphisms from $\CatCa_1$ to $\CatCa_2$.

If moreover $H$ is faithful, then $J$ is a complete invariant for $\CatCa_I$ and $J$ classifies homomorphisms from $\CatCa_1$ to $\CatCa_2$. In this case, we say that we can \emph{fully recover $I$ from $J$ through $H$}.
\end{thm}

\begin{proof}
Let $I,J$ and $H$ be functors as in the theorem.

(i) Suppose that $I$ is a complete invariant for $\CatCa_I$. Take any two $\CatCa$-algebras $A,B\in\CatCa_I$. If there exists an isomorphism $\alpha:J(A)\simeq J(B)$, by functoriality, we get an isomorphism $H(\alpha):H\circ J(A)\simeq H\circ J(B)$. Using the natural isomorphism $H\circ J\simeq I$, we know that $H(\alpha)$ gives us an isomorphism $\beta:I(A)\simeq I(B)$. By hypothesis, we can lift $\beta$ to an isomorphism in the category $\CatCa$. That is, there exists a $^*$-isomorphism $\phi:A\simeq B$ such that $I(\phi)=\beta$. We have just shown that $J$ is weakly-complete for $\CatCa_I$.

 Suppose now that $H$ is faithful. Then the natural isomorphism exactly gives us that $H\circ J(\phi)=H(\alpha)$. Now since $H$ is faithful, we conclude that $J(\phi)=\alpha$. That is, $J$ is a complete invariant for $\CatCa_I$.
 
(ii) Suppose that $I$ classifies homomorphisms from $A$ to $B$. Let $\alpha:J(A)\longrightarrow J(B)$ be any morphism in $\mathcal{D}$. If $\phi,\psi:A\longrightarrow B$ are $
^*$-homomorphisms such that $J(\phi)=J(\psi)=\alpha$, then composing with $H$, we get $H\circ J(\phi)=H\circ J(\psi)=H(\alpha)$. Thus, $I(\phi)=I(\psi)$, which gives us, by hypothesis, that $\phi\sim_{aue}\psi$. Hence $J$ weakly classifies homomorphisms from $A$ to $B$.

Finally if $H$ is faithful, then for any $\alpha:J(A)\simeq J(B)$, using again the natural isomorphism $H\circ J\simeq I$, we obtain: For any lift $\phi:A\longrightarrow B$ of $\beta:I(A)\longrightarrow I(B)$, where $\beta$ is the morphism obtained from $H(\alpha)$ as in the proof of (i) above, we have $H\circ J(\phi)=H(\alpha)$. Since $H$ is faithful, we get that $\alpha=J(\phi)$, from which we deduce that $J$ classifies homomorphisms from $A$ to $B$.
\end{proof}
We illustrate all the above with the following results:
\begin{prop}
By \autoref{thm:naturaltransfolma}, we can recover $\Cu$ and $\K_1$ from $\Cu_1$ through $\nu_+$ and $\nu_{max}$ respectively. As to be expected, neither $\nu_+$ nor $\nu_{max}$ are faithful functors. 
\end{prop}

\begin{proof}
Use the natural isomorphisms of \autoref{thm:naturaltransfolma}.
\end{proof}

\begin{cor}
Let $\phi,\psi:A\longrightarrow B$ be two $^*$-homomorphism between separable $\CatCa$-algebras of stable rank one. If $\Cu_1(\phi)=\Cu_1(\psi)$ then $\Cu(\phi)=\Cu(\psi)$ and $\K_1(\psi)=\K_1(\phi)$.
\end{cor}

\subsection{Recovering the \texorpdfstring{$\K_*$}{K*} invariant}
We now study a concrete use of \autoref{thm:recoverfunctor} to recover existing classifying functors from $\Cu_1$, and in the process, recall some classification results that have been obtained in the past. 
Here, we give some insight on $\K_*:=\K_0\oplus \K_1$. Although notations might slightly differ, all of this can be found in \cite{ELL93} and \cite{ELL96}.

An \emph{approximately homogeneous dimensional} algebra, written $\AH_d$ algebra, is an inductive limit of finite direct sums of the form $M_n(I_q)$ and $M_n(\mathcal{C}(X))$, where $I_q:=\{f\in M_q(\mathcal{C}([0,1]))\mid f(0),f(1)\in \mathbb{C}1_q\}$ is the \emph{Elliott-Thomsen dimension-drop interval algebra} and $X$ is one of the following finite connected CW complexes: $\{*\},\T,[0,1]$. Observe that we have the following inclusions: $\AF\subseteq \AI,\A\!\T\subseteq \AH_d\subseteq \1NCCW$, where $\1NCCW$ are the inductive limits of 1-dimensional non-commutative CW complexes (abbreviated one dimensional $\NCCW$ complexes).

The category of ordered groups with order-unit, written $\AbGp_{u}$, is the category whose objects are ordered groups with order-unit and morphisms are ordered group morphisms that preserve the order-unit.

\begin{dfn}(cf \cite[Definition 1.2.1]{ELL96})
Let $A$ be a (unital) $\CatCa$-algebra. We define $\K_*(A):=\K_0(A)\oplus \K_1(A)$. We also define ${\K_*(A)}_+:=\{([p]_{\K_0(A)},[v]_{\K_1(A)})\}\subseteq \K_0(A)\oplus \K_1(A)$, where $p$ is a projection in $A\otimes\mathcal{K}$ and $v$ is a unitary in the corner $p(A\otimes\mathcal{K})p$. Notice that we look at the $\K_1$ class of $v$ in $A$, that is, $[v+(1-p)]_{\K_1(A)}$. Finally, we define $1_{\K_*(A)}:=([1_A]_{\K_0},0_{\K_1})$.
\end{dfn}

\begin{prop}(cf \cite[\S 1.2.2]{ELL96})
Let $A$ and $B$ be unital $\CatCa$-algebras of stable rank one. Then \\$(\K_*(A),{\K_*(A)}_+)$ is an ordered group and $1_{\K_*(A)}\in {\K_*(A)}_+ $ is an order-unit of $\K_*(A)$. 

Thus, $(\K_*(A),{\K_*(A)}_+, 1_{\K_*(A)})\in \AbGp_{u}$.
Moreover, any $^*$-homomorphism $\phi:A\longrightarrow B$ induces an ordered group morphism $\K_0(\phi)\oplus \K_1(\phi): \K_*(A)\longrightarrow \K_*(B)$ that preserves the order-unit. Thus, we obtain a covariant functor $\K_*:AH_{d,1}\longrightarrow \AbGp_{u}$, where $AH_{d,1}$ is the category of unital $AH_d$ algebras.
\end{prop}

We do not give a proof of the above, but we remind the reader that whenever a $\CatCa$-algebra $A$ has stable rank one -which is the case of any $\AH_d$ algebra-, then the monoid $V(A)$ has cancellation and hence ${\K_0(A)}_+$ can be identified with $V(A)$ and thus $(\K_0(A),V(A))$ is an ordered group. 

We also recall that in the stable rank one case, the Murray-von Neumann equivalence and the Cuntz equivalence agree on the projections of $A\otimes\mathcal{K}$ and that $V(A)\simeq \Cu(A)_c$. That is, any compact element of $\Cu(A)$ is the class of some projection of $A\otimes\mathcal{K}$.

We now recall two notable classification results by means of $\K_*$ that catch our interest:

\begin{thm}(\cite[Corollary 4.9]{ELL96}, \cite[Theorem 7.3 - Theorem 7.4]{ELL93})

(i) The functor $\K_*$ is a complete invariant for (unital) $AH_{d}$ algebras of real rank zero.

(ii) Let $A,B$ be (unital) $\A\!\T$ algebras of real rank zero and let $\alpha:\K_*(A)\longrightarrow \K_*(B)$ be a scaled ordered group morphism. Then there exists a unique $^*$-homomorphism (up to approximate unitary equivalence) $\phi:A\longrightarrow B$ such that $\K_*(\phi)=\alpha$.
\end{thm}

The aim now is to recover $\K_*$ from $\Cu_1$ and thus show that $\Cu_1$ contains more information than $\K_*$. For that purpose, we first define the category of $\Cu^\sim$-semigroups with order-unit, that we denote by $\Cu^\sim_{u}$. Further, we create a functor $H_*:\Cu^\sim_{u}\longrightarrow \AbGp_{u}$ such that $H_*\circ \Cu_1\simeq \K_*$ as functors. Moreover, restricting to an adequate subcategory of $\Cu^\sim_{u}$, we will see that $H_*$ is faithful.

\begin{dfn}
Let $S$ be a $\Cu^\sim$-semigroup. We say that $S$ has \emph{weak cancellation} if $x+z\ll y+z$ implies $x\leq y$ for $x,y,z\in S$. We say that $S$ has \emph{cancellation of compact elements} if $x+z\leq y+z$ implies $x\leq y$ for any $x,y\in S$ and $z\in S_c$.
\end{dfn}

The following property is proved using the same argument as in \cite[Proposition 2.1.3]{RobNCCW1}.
\begin{prop}
\label{cor:wcsr1} 
Let $A$ be a $\CatCa$-algebra of stable rank one. Then $\Cu_1(A)$ has weak cancellation and a fortiori $\Cu_1(A)$ has cancellation of compact elements.
\end{prop}

\begin{dfn}
Let $S$ be a positively directed $\Cu^\sim$-semigroup. Suppose that $S$ has cancellation of compact elements. Also suppose that $S_+$ admits a compact order-unit. 

We say that $(S,u)$ is a \emph{$\Cu^\sim$-semigroup with compact order-unit}. Now, a \emph{$\Cu^\sim$-morphism} between two $\Cu^\sim$-semigroups with compact order-unit $(S,u),(T,v)$ is a $\Cu^\sim$-morphism $\alpha:S\longrightarrow T$ such that $\alpha(u)\leq v$.

We define the category of $\Cu^\sim$-semigroups with compact order-unit, denoted $\Cu^\sim_{u}$, as the category whose objects are $\Cu^\sim$-semigroups with order-unit and morphisms are $\Cu^\sim$-morphisms that preserve the order-unit.
\end{dfn}

\begin{lma}
The assignment \vspace{-0,7cm}\[
	\begin{array}{ll}
	\hspace{0,4cm} \Cu_{1,u}: \CatCa_{sr1,1}\longrightarrow \Cu^\sim_{u} \\
	\hspace{2cm} A\longmapsto (\Cu_1(A),([1_A],0))\\
		\hspace{2cm} \phi \longmapsto \Cu_1(\phi)
	\end{array}
\] 
from the category of unital $\CatCa$-algebras of stable rank one, denoted by $\CatCa_{sr1,1}$, to the category $\Cu^\sim_u$ is a covariant functor.
\end{lma}

\begin{proof}
We know that $\Cu_1(A)_+$ has cancellation of compact elements. Further, we know that $([1_A],0)$ is a compact order-unit of $\Cu_1(A)_+$, so it easily follows that $\Cu_{1,u}(A)\in \Cu^\sim_{u}$. Finally, it is trivial to see that $\Cu_1(\phi)([1_A])\leq [1_B]$, which ends the proof.
\end{proof}

\begin{lma}
\label{lma:NIKISONMIFIT}
The assignment
\vspace{-0,7cm}\[
	\begin{array}{ll}
	\hspace{1,7cm} H_*:\Cu^\sim_{u}\longrightarrow \AbGp_{u}\\
	\hspace{2,2cm} (S,u)\longmapsto (\Gr(S_{c}),S_{c},u)\\
		\hspace{2,8cm} \alpha \longmapsto \Gr(\alpha_{c})
	\end{array}
\] 
from the category $\Cu^\sim_{u}$ to the category $\AbGp_{u}$ is a covariant functor.

Moreover, if we restrict the domain of $H_*$ to the category of algebraic $\Cu^\sim_u$-semigroups with compact order-unit, denoted by $\Cu^\sim_{u,alg}$, then $H_*$ becomes a faithful functor.
\end{lma}

\begin{proof}
Let $(S,u)\in\Cu^\sim_{u}$. By \autoref{cor:wcsr1}, we know that $S_c$ is a monoid with cancellation and hence, using the Grothendieck construction, one can check that $(\Gr(S_{c}),S_{c},u)$ is an ordered group with order-unit.
Now let $\alpha:S\longrightarrow T$ be a $\Cu^\sim_u$-morphism between two $\Cu^\sim$-semigroups with order-unit $(S,u),(T,v)$. By functoriality of $\nu_c$, it follows that $\alpha_{c}:S_{c}\longrightarrow T_{c}$ is a $\CatoM$-morphism, and hence that $\Gr(\alpha_{c}):\Gr(S_{c})\longrightarrow \Gr(T_{c})$ is a group morphism such that $\Gr(\alpha_{c})(S_{c})\subseteq T_{c}$. Finally, using that $\alpha(u)\leq \alpha(v)$, we obtain $\Gr(\alpha_{c})(u)\leq v$. We conclude that $H_*$ is a well-defined functor.

Now, we have to show that if we restrict the domain of $H_*$ to $\Cu^\sim_{sc,alg}$, then $H_*$ becomes faithful. Let $\alpha,\beta:(S,u)\longrightarrow (T,v)$ be two scaled $\Cu^\sim$-morphisms between $(S,u),(T,v)\in \Cu^\sim_{sc,alg}$ such that $H_*(\alpha)=H_*(\beta)$. In particular, $\alpha_c=\beta_c$, and since we are in the category of algebraic $\Cu^\sim$-semigroups, any element is the supremum of an increasing sequence of compact elements. Thus any morphism is entirely determined by its restriction to compact elements. One can conclude that $\alpha=\beta$ and the proof is complete.
\end{proof}

\begin{thm}
The functor $H_*:\Cu^\sim_{u}\longrightarrow \AbGp_{u}$ yields a natural isomorphism $\eta_*:H_*\circ\Cu_{1,u}\simeq \K_*$.
\end{thm}

\begin{proof}
First we prove that ${\K_*(A)}_+\simeq {\Cu_1(A)}_c$ as monoids and the result will follow from the Grothendieck construction.

We know that ${\Cu_1(A)}_c$ is a monoid. Now consider $[(a,u)]\in\Cu_1(A)_c$. By \autoref{cor:compactcu1}, we know that $[a]$ is a compact element of $\Cu(A)$. Besides, since $A$ has stable rank one, we know that we can find a projection $p\in A\otimes\mathcal{K}$ such that $[p]=[a]$ in $\Cu(A)$. So without loss of generality, we now describe compact elements of $\Cu_1(A)$ as classes $[(p,u)]$ where $p$ is projection in $A\otimes\mathcal{K}$ and $u$ is a unitary element in $\her (p)$.

On the other hand, by \autoref{thm:naturaltransfolma}, we have $\Cu_1(A)_{max}\simeq \K_1(A)$, where the $\AbGp$-isomorphism is given by $[(s_{A\otimes\mathcal{K}},u)]\longmapsto [u]$, where $s_{A\otimes\mathcal{K}}$ is any strictly positive element of $A\otimes\mathcal{K}$. Combined with \autoref{dfn:vch}, we get a monoid morphism $j:\Cu_1(A)\longrightarrow \K_1(A)$. Now set:
\[
\begin{array}{ll}
\alpha:\Cu_1(A)_c\longrightarrow {\K_*(A)}_+\\
\hspace{0,7cm}[(p,u)]\longmapsto ([p],j([p,u]))
\end{array}
\]
It is routine to check that $\alpha$ is monoid morphism. Further, observe that $j([p,u])=\delta_{I_pA}([u])$ for any $[(p,u)]\in \Cu_1(A)_c$, where $\delta_{I_pA}:\K_1(\her (p))\overset{\K_1(i)}\longrightarrow \K_1(A)$ (see \autoref{dfn:delta}). Thus, $j([p,u])=[u+(1-p)]_{\K_1(A)}$. Now, since $A$ has stable rank one, Murray-von Neumann equivalence and Cuntz equivalence agree on projections. It is now clear that $\alpha$ is an isomorphism and hence $\Cu_1(A)_c\simeq {\K_*(A)}_+$ as monoids. From the Grothendieck construction, one can check that $(\K_*(A),{\K_*(A)}_+)\simeq (\Gr(\Cu_1(A)_c),\Cu_1(A)_c)$ as ordered groups. Finally, it is routine to check that $[(1_A,1_A)]$ is a compact order-unit for $\Cu_1(A)$ (a fortiori, an order-unit for $(\Gr(\Cu_1(A)_c),\Cu_1(A)_c)$) and that $\alpha([(1_A,1_A)])=1_{\K_*(A)}$. 

We conclude that for any $A\in \CatCa_{sr1,1}$, there exists a natural ordered group isomorphism ${\eta_*}_A:H_*\circ\Cu_{1,u}(A)\simeq (\K_*(A),{\K_*(A)}_+,1_{\K_*(A)})$ that preserves the order-unit and hence there exists a natural isomorphism $\eta_*: H_*\circ\Cu_{1,u}\simeq \K_*$. 
\end{proof}

\begin{cor}
By restricting to the category $\Cu^\sim_{u,alg}$, we can fully recover $\K_*$ from $\Cu_{1,u}$ through $H_*$. A fortiori, we have:

(i) $\Cu_{1,u}$ is a complete invariant for unital $AH_{d}$ algebras of real rank zero.

(ii) $\Cu_{1,u}$ classifies homomorphisms of unital $\A\!\T$ algebras with real rank zero.
\end{cor}

\end{document}